\documentclass[a4paper, 12pt]{article}

\usepackage[sort&compress]{natbib}
\bibpunct{(}{)}{;}{a}{}{,} 

\usepackage{amsthm, amsmath, amssymb, mathrsfs, multirow, url, subfigure}
\usepackage{graphicx} 
\usepackage{ifthen} 
\usepackage{amsfonts}
\usepackage[usenames]{color}
\usepackage{fullpage}

\numberwithin{equation}{section} 

\theoremstyle{plain} 
\newtheorem{theorem}{Theorem}
\newtheorem{corollary}{Corollary}
\newtheorem{proposition}{Proposition}
\newtheorem{lemma}{Lemma}

\theoremstyle{definition}

\theoremstyle{remark}
\newtheorem{example}{Example}
\newtheorem{remark}{Remark}

\theoremstyle{remark} 
\newtheorem*{astep}{A-step}
\newtheorem*{pstep}{P-step}
\newtheorem*{cstep}{C-step}

\newcommand{\prob}{\mathsf{P}} 
\newcommand{\E}{\mathsf{E}}

\newcommand{\bel}{\mathsf{bel}}
\newcommand{\pl}{\mathsf{pl}}
\newcommand{\cbel}{\mathsf{cbel}}
\newcommand{\cpl}{\mathsf{cpl}}

\newcommand{\pois}{{\sf Pois}}
\newcommand{\unif}{{\sf Unif}}
\newcommand{\nm}{{\sf N}}
\newcommand{\expo}{{\sf Exp}}
\newcommand{\gam}{{\sf Gam}}
\newcommand{\stt}{{\sf t}}

\newcommand{\ber}{{\sf Ber}}

\newcommand{\chisq}{{\sf ChiSq}}

\newcommand{\RR}{\mathbb{R}}

\newcommand{\XX}{\mathbb{X}}
\newcommand{\UU}{\mathbb{U}}
\newcommand{\VV}{\mathbb{V}}

\newcommand{\G}{\mathscr{G}}
\newcommand{\Gbar}{\overline{\mathscr{G}}}
\newcommand{\gbar}{\overline{g}}

\newcommand{\Xbar}{\overline{X}}

\newcommand{\Ubar}{\overline{U}}

\renewcommand{\S}{\mathcal{S}}
\renewcommand{\SS}{\mathbb{S}}

\newcommand{\ustar}{u^\star}
\newcommand{\vstar}{v^\star}

\newcommand{\del}{\partial}
\DeclareMathOperator{\rank}{rank}
\newcommand{\diag}{\mathrm{diag}}

\renewcommand{\phi}{\varphi} 
\newcommand{\eps}{\varepsilon}

\newcommand{\iid}{\overset{\text{\tiny iid}}{\,\sim\,}}

\title{Conditional inferential models: combining information for prior-free probabilistic inference}
\author{
Ryan Martin \\
Department of Mathematics, Statistics, and Computer Science \\
University of Illinois at Chicago \\
\url{rgmartin@uic.edu} \\
\mbox{} \\
Chuanhai Liu \\
Department of Statistics \\
Purdue University \\
\url{chuanhai@purdue.edu}
}
\date{\today}

\begin{document}

\maketitle 

\begin{abstract}
The inferential model (IM) framework provides valid prior-free probabilistic inference by focusing on predicting unobserved auxiliary variables.  But, efficient IM-based inference can be challenging when the auxiliary variable is of higher dimension than the parameter.  Here we show that features of the auxiliary variable are often fully observed and, in such cases, a simultaneous dimension reduction and information aggregation can be achieved by conditioning.  This proposed conditioning strategy leads to efficient IM inference, and casts new light on Fisher's notions of sufficiency, conditioning, and also Bayesian inference.  A differential equation-driven selection of a conditional association is developed, and validity of the conditional IM is proved under some conditions.  For problems that do not admit a valid conditional IM of the standard form, we propose a more flexible class of conditional IMs based on localization.  Examples of local conditional IMs in a bivariate normal model and a normal variance components model are also given.   

\smallskip

\end{abstract}

\section{Introduction}
\label{S:intro}

Fisher's brand of statistical inference \citep{fisher1973} is often viewed as a middle-ground between the Bayesian and frequentist approaches.  Two important examples are his fiducial argument and his ideas on conditional inference.  Perhaps influenced by Fisher's ideas, a current focus in foundational research is on achieving some kind of compromise between the Bayesian and frequentist ideals.  See, for example, recent work on fiducial inference \citep{hannig2009, hannig2012, hannig.lee.2009}, confidence distributions \citep{xie.singh.strawderman.2011, xie.singh.2012}, Dempster--Shafer theory \citep{dempster2008, shafer2011}, and objective Bayes with default, reference, and/or data-dependent priors \citep{berger2006, bergerbernardosun2009, fraser2011, fraser.reid.marras.yi.2010}.  Recently \citet{imbasics} have laid out the details of a promising new \emph{inferential model} (IM) approach.  IMs take the usual input---sampling model and observed data---and produce prior-free, probabilistic measures of certainty about any assertion/hypothesis of interest, with an almost automatic calibration property.  The fundamental idea is that uncertainty about the parameter of interest $\theta$, given observed data $X=x$, is fully characterized by the of an unobservable auxiliary variable $U$.  So, the problem of inference about $\theta$ can be translated into one of predicting this unobserved $U$ with a random set.  In Section~\ref{S:ims} we briefly review the construction and basic properties of IMs. 

The discussion in \citet{imbasics} focuses on the case where $\theta$ and $U$ are of the same dimension.  But there are many problems, e.g., iid data from scalar parameter models, where the dimension of the auxiliary variable is greater than that of the parameter.  In such cases, efficiency can be gained by first reducing the dimension of the auxiliary variable to be predicted, though it is not obvious how this should be done in general.  Here we focus our attention on an auxiliary variable dimension reduction step based on conditioning.  The key observation is that, typically, certain functions of the auxiliary variables are fully observed.  By conditioning on those observed characteristics of the auxiliary variable, we can effectively reduce the dimension of the unobserved characteristics to be predicted.  A motivating example, demonstrating the efficiency gain from dimension reduction, along with the detailed developments are presented in Section~\ref{S:conditional}.  The proposed dimension-reduction approach, based on conditioning, can be viewed as a general tool for combining information about $\theta$ across samples---a counterpart to Bayes' theorem and sufficiency.  With the lower-dimensional auxiliary variable, we proceed to construct what is called a \emph{conditional IM}.  We prove a validity theorem that establishes a desirable calibration property of the conditional IM, and facilitates a common interpretation across users and experiments.  

Finding the dimension-reduced representation is sometimes a familiar task.  For example, when the minimal sufficient statistic has dimension matching that of the parameter, the conditional IM is exactly that obtained by working directly with said statistic.  In other cases, finding the lower-dimensional representation is not so simple.  For this, in Section~\ref{S:finding}, we propose a new differential equation-driven technique for identifying observed characteristics of the auxiliary variable.  Two classical conditional inference problems are worked out in Section~\ref{S:more-examples}, one showing how the proposed differential equation technique leads to an additional dimension reduction beyond what ordinary sufficiency provides.  So, besides the development of conditional IMs, the proposed framework also casts new light on the familiar notion of sufficiency, as well as Fisher's attractive but elusive ideas on ancillary statistics and conditional inference.     

In some cases, however, it may not be possible to produce a valid conditional IM with these somewhat standard techniques.  For this, we propose an extension of the conditional IM framework, in Section~\ref{S:gcim}, which allows the lower-dimensional auxiliary variable representation to depend on $\theta$ in a certain sense.  We refer to these as \emph{local conditional IMs}, and we describe their construction and prove a validity theorem.  An important example of such a problem is the bivariate normal model with known means and variances but unknown correlation.  For this example, we construct a local conditional IM based on a modification of our differential equations technique, and provide the results of a simulation study that shows that our conditional plausibility intervals outperform the classical $r^\star$-driven asymptotically approximate confidence intervals \citep{bn1986, fraser1990} in both small and large samples.  A local conditional IM analysis of the variance-components model, another benchmark problem, is also given.  

\section{Review of IMs}
\label{S:ims}

\subsection{Notation and construction}
\label{SS:notation}

To fix notation, let $X$ be the observable data, taking values in a space $\XX$, and let $\theta$ be the parameter of interest, taking values in the parameter space $\Theta$.  The starting point of the IM framework is similar to that of fiducial, in the sense that an auxiliary variable, denoted by $U$ and taking values in a space $\UU$ with probability measure $\prob_U$, is associated with $X$ and $\theta$.  It is this association, together with the distribution $U \sim \prob_U$, that characterizes the sampling distribution $X \sim \prob_{X|\theta}$.  In particular, if we write this association as 
\begin{equation}
\label{eq:assoc}
X = a(\theta,U), \quad U \sim \prob_U.
\end{equation}
Throughout, subscripts on $\prob$ indicate which quantity is random.  

Fiducial inference employs the sampling distribution $\prob_U$ \emph{after} $X=x$ is observed.  The IM approach is different in that it treats the unobserved value of $U$, which is tied to the observed data $X=x$ and the \emph{true value} of $\theta$, as the fundamental quantity.  Then the goal is to predict this unobserved value with a random set before conditioning on $X=x$ and inverting \eqref{eq:assoc}.  Let $(\UU,\mathscr{U},\prob_U)$ be a probability space, where $\mathscr{U}$ is rich enough to contain all closed subsets of $\UU$.  Take a collection $\SS$ of closed (hence $\prob_U$-measurable) subsets of $\UU$, assumed to contain $\varnothing$ and $\UU$.  This collection will serve as the support of the predictive random set.  We shall also assume that the collection $\SS$ is nested, i.e., either $S \subseteq S'$ or $S' \subseteq S$ for all $S,S' \in \SS$.  Now define the predictive random set $\S \sim \prob_\S$, supported on $\SS$, with distribution $\prob_\S$ satisfying 
\[ \prob_\S\{\S \subseteq K\} = \sup_{S \in \SS: S \subseteq K} \prob_U(S), \quad K \subseteq \UU. \]
Predictive random sets with these properties are called \emph{admissible}. \citet{imbasics} 
give a sort of complete-class theorem for admissible predictive random sets.  In scalar $\theta$ problems, $\prob_U $ is often $\unif(0,1)$, so an important example of an admissible predictive random set is 
\begin{equation}
\label{eq:default.prs}
\S = \{u: |u-\tfrac12| \leq |U-\tfrac12|\}, \quad U \sim \unif(0,1).
\end{equation}
\citet[][Corollary~1]{imbasics} show that this $\S$---called the ``default'' predictive random set---has a variety of good properties, and these good properties often carry over to the corresponding IM.  It should be mentioned that admissible predictive random sets are not unique.  In this paper we focus primarily on simplicity, but the ideas on optimal predictive random sets in \citet{imbasics} can also be applied here.  

The following three steps---association, predict, and combine---define an IM.  

\begin{astep}
Associate $X$, $\theta$, and $U \sim \prob_U$, consistent with the sampling distribution $X \sim \prob_{X|\theta}$, such that, for all $(x,u)$, there is a unique subset $\Theta_x(u) = \{\theta: x=a(\theta,u)\} \subseteq \Theta$, possibly empty, containing all possible candidate values of $\theta$ given $(x,u)$.  
\end{astep}

\begin{pstep}
Predict the unobserved value $u^\star$ of $U$ associated with the observed data by an admissible predictive random set $\S$.  
\end{pstep}

\begin{cstep}
Combine $\S$ and the association $\Theta_x(u)$ specified in the A-step to obtain 
\begin{equation}
\label{eq:new.focal}
\Theta_x(\S) = \bigcup_{u \in \S} \Theta_x(u).
\end{equation}
Then compute the \emph{belief function}
\begin{equation}
\label{eq:belief}
\bel_x(A; \S) = \prob_\S\{\Theta_x(\S) \subseteq A \mid \Theta_x(\S) \neq \varnothing\}, 
\end{equation}
where $A \subseteq \Theta$ is the assertion/hypothesis about $\theta$ of interest.  \citet{leafliu2012} give an alternative to conditioning on the event ``$\Theta_x(\S) \neq \varnothing$,'' based on stretching, that tends to be more efficient.  
\end{cstep}

The belief function is just one part of the inferential output.  Since the belief function is sub-additive, i.e., $\bel_x(A;\S) + \bel_x(A^c;\S) \leq 1$, one actually needs both $\bel_x(A;\S)$ and $\bel_x(A^c;\S)$ to summarize the information in $x$ concerning the truthfulness of assertion $A$.  In some cases, it is more convenient to report the \emph{plausibility function} 
\begin{equation}
\label{eq:plausibility}
\pl_x(A;\S) = 1-\bel_x(A^c; \S).
\end{equation}
Then the pair $(\bel_x,\pl_x)(A;\S)$ characterize the IM output.  


The IM and fiducial approaches both start with a representation of the sampling model using auxiliary variables, but, beyond that, there are some important differences.  First, by taking the predictive random set $\S = \{U\}$, with $U \sim \prob_U$, a random singleton, the corresponding IM is exactly the fiducial distribution for $\theta$.  However, this singleton random set is not nested, hence, not admissible, so the desirable validity properties in Section~\ref{SS:validity} are not guaranteed.  Second, a subtle point, the interpretation of probability changes as one proceeds along the fiducial argument.  One starts with a non-subjective probability $\prob_U$ for $U$ before $X=x$ is observed.  Then, after $X=x$ is observed, the conditional distribution of $U$, which is concentrated on the set $\{u: x=a(\theta,u)\}$, where $\theta$ is the fixed true parameter value, is replaced by the original $\prob_U$, i.e., one ``continues to regard'' $U$ as a sample from $\prob_U$ after $X=x$ is observed \citep{dempster1963}.  Therefore, despite starting with a non-subjective probability $\prob_U$, the choice to replace the conditional distribution of $U$, given $X=x$, with $\prob_U$ ultimately makes the fiducial probabilities subjective.  This explains why fiducial inference is not valid for some assertions (and finite $n$).  IM probabilities, on the other hand, are based on $\prob_\S$, and the fact that there are no direct links between data and $\S$ means that the $\prob_\S$-probabilities are the same---in terms of both computation and interpretation---before and after $X=x$ is observed.  This explains why IM-based inference is valid whenever the predictive random set is suitably calibrated to $\prob_U$; see Section~\ref{SS:validity} below.

Finally, without practical loss of generality, assume that $\{\prob_{X|\theta}: \theta \in \Theta\}$ has a common dominating measure, say $\mu$.  Then we require that $\bel_x(A;\S)$ be a $\mu$-measurable function in $x$.  This is easy to check in examples, but general sufficient conditions are more elusive.  To keep the presentation simple, we shall mostly ignore these technicalities.

\ifthenelse{1=1}{}{
\subsection{An illustrative example}
\label{SS:gamma.example}

Let $X$ be a single gamma observation with unknown shape $\theta > 0$, i.e., $X \sim \gam(\theta,1)$.  If $F_\theta$ denotes this gamma distribution function, then we can the association as $X = F_\theta^{-1}(U)$, where $U \sim \unif(0,1)$.  The result of the A-step is the simple singleton valued mapping $\Theta_x(u) = \{\theta: F_\theta(x)=u\}$, for given $X=x$.  For the P-step, to keep things simple, we shall use the default predictive random set $\S$ in \eqref{eq:default.prs}, characterized by a draw $U \sim \unif(0,1)$.  Finally, for the C-step, we combine $\Theta_x(\cdot)$ and $\S$ to get the random set 
\[ \Theta_x(\S) = \{\theta: |F_\theta(x)-\tfrac12| \leq |U-\tfrac12|\}, \quad U \sim \unif(0,1). \]
The belief and plausibility functions can be readily calculated based on the distribution of $\Theta_x(\S)$.  For example, for an interval assertion $A_s = \{\theta \leq s\}$, one can get 
\begin{align*}
\bel_x(A_s) & = \prob_\S\{\Theta_x(\S) \subseteq A_s\} = \prob_\S\{\sup \Theta_x(\S) \leq s\} = \max\{1-2F_s(x),0\}, \\
\pl_x(A_s) & = \prob_\S\{\Theta_x(\S) \not\subseteq A_s^c\} = \prob_\S\{\inf \Theta_x(\S) \leq s\} = 1 - \max\{2F_s(x)-1,0\}. 
\end{align*}
Similarly, if $A_s=\{\theta=s\}$ is a singleton, then $\bel_x(A_s) = 0$ for all $s$, and 
\[ \pl_x(A_s) = 1 - |2F_s(x)-1|. \]
For observed $X=5$, graphs of these functions, as $s$ varies, are shown in Figure~\ref{fig:gamma.example}.  

\begin{figure}
\centering
\fbox{
\subfigure[Assertion $A_s=\{\theta \leq s\}$.]{\scalebox{0.5}{\includegraphics{gamma_int}}}
\subfigure[Assertion $A_s=\{\theta=s\}$.]{\scalebox{0.5}{\includegraphics{gamma_point}}}
}
\caption{\label{fig:gamma.example} Belief and plausibility functions for the gamma shape parameter problem in Section~\ref{SS:gamma.example}, with $X=5$ observed, as a function of the assertion parameter $s$.  Panel~(a): belief function (solid) and plausibility function (dashed).  Panel~(b), plausibility function (solid) and belief function (not shown) is zero.}
\end{figure}

}

\subsection{Validity of IMs}
\label{SS:validity}


Given $\S$, the corresponding IM is \emph{valid for} $A$ if the belief function satisfies
\begin{equation}
\label{eq:bel.valid}
\sup_{\theta \not\in A} \prob_\theta\bigl\{ \bel_X(A;\S) \geq 1-\alpha \bigr\} \leq \alpha, \quad \forall\;\alpha \in (0,1).
\end{equation}
The IM is simply called \emph{valid} if it is valid for all $A$.  In other words, the IM is valid for $A$ if $\bel_X(A;\S)$ is stochastically no larger than $\unif(0,1)$ when $X \sim \prob_{X|\theta}$ with $\theta \not\in A$.  That is, if $A$ is false, then the amount of support in data $X$ for $A$ will be large only for a relatively small proportion of $X$ values.  \citet{imbasics, imbasics.c} show that this validity property holds whenever the predictive random set $\S$ is admissible.  If the IM is valid for all $A$, then \eqref{eq:bel.valid} can be equivalently stated in terms of the plausibility function:
\begin{equation}
\label{eq:pl.valid}
\sup_{\theta \in A} \prob_{X|\theta}\{\pl_X(A; \S) \leq \alpha\} \leq \alpha, \quad \forall \; \alpha \in (0,1).  
\end{equation}
This formulation is occasionally more convenient than \eqref{eq:bel.valid}. 

There are two important consequences of the validity theorem.  First, it helps determine an objective scale on which the belief probabilities can be interpreted.  Unlike valid IMs, the output from default-prior Bayesian, fiducial, and Dempster--Shafer inference does not have a specified scale for interpretation.  For example, is a Bayesian or fiducial posterior probability of 0.9 a large value?  It is common to think on the usual frequency scale, i.e., betting on an event with 0.9 probability wins 90\% of the time, but there is no justification for this without some notion of validity as in \eqref{eq:bel.valid} or \eqref{eq:pl.valid}.  Second, validity justifies the use of IM output to construct frequentist decision procedures with control on error rates.  For example, a $100(1-\alpha)$\% plausibility region for $\theta$ is given by
\begin{equation}
\label{eq:plaus.region}
\{\theta: \pl_x(\theta; \S) > \alpha\}. 
\end{equation}
It follows easily from \eqref{eq:pl.valid} that this plausibility region has nominal $1-\alpha$ coverage probability.  But we should emphasize here that, although plausibility functions can be used to construct frequentist procedures, the interpretation is quite different.  For example,  the plausibility region is understood as the collection of $\theta$'s such that each is \emph{individually} sufficiently plausible, given $X=x$.  Confidence/credible regions, on the other hand, say nothing about the plausibility of any particular $\theta$ they contain.  


\section{Conditional IMs}
\label{S:conditional}

\subsection{Motivation}
\label{SS:normal.example}

In the case of a scalar auxiliary variable, construction of efficient predictive random sets is relatively easy.  However, rarely does the model directly admit a scalar auxiliary variable representation.  To see this, suppose $X_1,\ldots,X_n$ are independent $\nm(\theta,1)$ with unknown mean $\theta$.  In vector notation, an association is $X = \theta 1_n + U$, where $1_n$ is an $n$-vector of unity, and $U \sim {\sf N}_n(0,I)$.  At first look, it seems that one must predict an $n$-dimensional auxiliary variable $U$.  But efficient prediction of $U$ would be challenging, even for moderate $n$, so reducing its dimension---ideally to one---would be a desirable first step.  After reducing the dimension to one, choosing efficient predictive random sets is as easy as in the scalar auxiliary variable case considered in \citet{imbasics}.  

The basic point is that one pays a price, in terms of efficiency, for predicting higher-dimensional auxiliary variables.  To see this better, we shall take a closer look at the normal mean problem above, with $n=2$.  That is, we have a baseline association 
\[ X_1 = \theta + U_1 \quad \text{and} \quad X_2 = \theta + U_2, \]
where $U_1,U_2$ are independent $\nm(0,1)$.  To make things simple, consider the following change of variables: $Y_1 = X_1+X_2$ and $Y_2 = X_1-X_2$.  In the new variables, we have 
\[ Y_1 = 2\theta + V_1 \quad \text{and} \quad Y_2 = V_2, \]
where $V_1,V_2$ are independent $\nm(0,2)$.  This completes the A-step.  Following the basic procedure described in Section~\ref{S:ims}, for the P-step, we should predict the pair $(V_1,V_2)$ with a predictive random set $\S$.  A simple $L_\infty$ generalization of the default predictive random set \eqref{eq:default.prs} to the case of a two-dimensional auxiliary variable is a random square: 
\[ \S = \{(v_1,v_2): \max(|v_1|,|v_2|) \leq \max(|V_1|,|V_2|)\}, \quad V_1,V_2 \iid \nm(0,2). \]
For a singleton assertion $\{\theta\}$, the C-step gives plausibility function
\[ \pl_y(\theta) = \frac{1 - G( 2^{-1/2}\max\{|y_1-2\theta|,|y_2|\} )^2}{1 - G( 2^{-1/2}|y_2| )^2}, \] 
where $G(z)=1-2(1-\Phi(z))$ is the $|\nm(0,1)|$ distribution function.  The unusual form here is due to the conditioning in \eqref{eq:belief} to remove conflict cases where $\Theta_y(\S) = \varnothing$.  

As an alternative approach, note that the value of $V_2$ is known once $Y_2$ is observed.  So rather than trying to predict this component, as in the approach just described, we might condition on this observed value, to sharpen our uncertainty for predicting $V_1$.  Since $V_1$ and $V_2$ are actually independent, it suffices to work with the marginal distribution, $V_1 \sim \nm(0,2)$.  For the A-step, we get $Y_1 = 2\theta + V_1$ and, for the P-step, we use a default predictive random set $\S=\{v_1: |v_1| \leq V_1\}$, where $V_1 \sim \nm(0,2)$.  For the same singleton assertion, the C-step this time gives plausibility function 
\[ \pl_y(\theta) = 1 - |2\Phi\bigl(2^{-1/2}(y_1-2\theta)\bigr) - 1|. \]

The claim is that inference based on the latter IM formulation is more efficient than that based on the former.  To check this, we consider the sampling distribution of $\pl_Y(0)$ in the case where $Y=(Y_1,Y_2)$ is an independent $\nm(0,2)$ random vector.  Figure~\ref{fig:plaus.qq} shows a quantile plot of the two simulated samples.  By the validity theorem, the plausibilities are both stochastically no smaller than $\unif(0,1)$.  However, we see that plausibilities for the reduced, one-dimensional predictive random set are exactly $\unif(0,1)$ distributed, while those based on the two-dimensional predictive random set tend to be considerably larger.  The larger plausibility means less efficiency, e.g., wider plausibility intervals, so the IM based on the reduced one-dimensional predictive random set is preferred.  This difference in efficiency is explained by the fact that the two-dimensional predictive random set for $(V_1,V_2)$ corresponds to a larger-than-necessary predictive random set for $V_1$; the conflict cases in the two-dimensional case have little to no effect on efficiency.

\begin{figure}
\centering
\fbox{
\scalebox{0.5}{\includegraphics{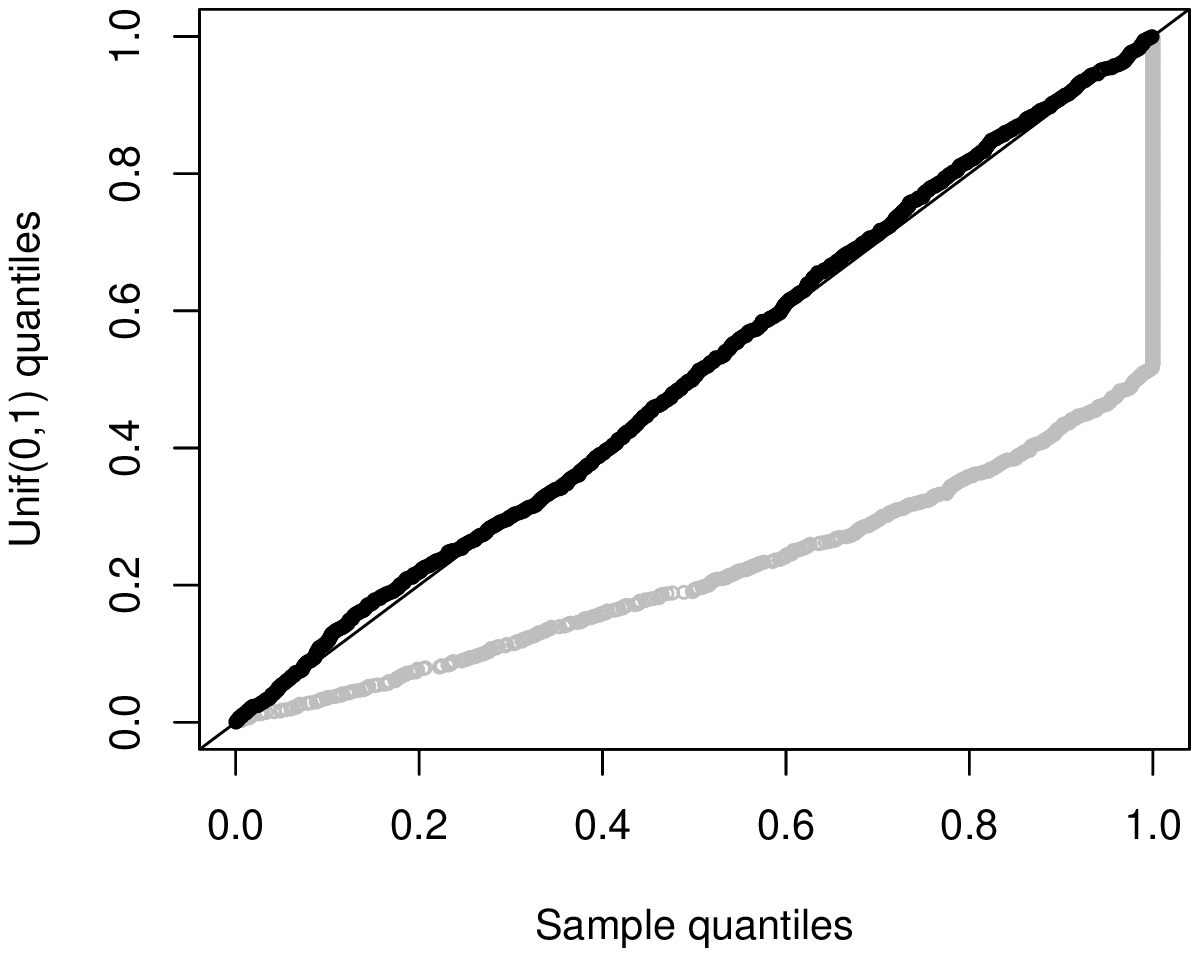}}
}
\caption{\label{fig:plaus.qq} Quantile plot of the two plausibility functions $\pl_Y(0)$ defined in Section~\ref{SS:normal.example}.  Gray points correspond to the two-dimensional predictive random set; black points correspond to the one-dimensional predictive random set.}
\end{figure}

In the remainder of this section, we will give a general prescription for increasing efficiency by reducing the dimension.  The key is that, in general, some functions of the original auxiliary variable are fully observed, like $V_2=U_1-U_2$ in this simple example.  Then the strategy is to condition on what is fully observed to sharpen prediction of what is not observed.  Since this ``conditioning to sharpen inference'' strategy is commonly used in statistics, similar considerations are natural in the IM framework.

\subsection{Dimension reduction via conditioning}
\label{SS:cim}

Here we propose a conditioning strategy, whereby a simultaneous information aggregation and dimension reduction is achieved, that results in an overall gain in efficiency.  The intuition is that some functions of the unobserved $\ustar$ are actually observed, so these characteristics do not need to be predicted.  Focusing only on the unobserved characteristics of $\ustar$ leads directly to a dimension reduction.  However, knowledge about the observed characteristics helps to better predict those unobserved characteristics, so information is accumulated and prediction is sharpened.  The general strategy is as follows:
\begin{itemize}
\item Identify an observed characteristic, $\eta(U)$, of the auxiliary variable $U$ whose distribution is free (or at least mostly free) of $\theta$, and 
\vspace{-2mm}
\item define a conditional association that relates an unobserved characteristic, $\tau(U)$, of the auxiliary variable $U$ to $\theta$ and some function $T(X)$ of the data $X$.  
\end{itemize}
The second step is familiar, as it relates to working with, say, a minimal sufficient statistic.  The first step, however, is less familiar and generally more difficult; see Section~\ref{S:finding}.  

To make this formal, suppose that $x \mapsto (T(x),H(x))$ and $u \mapsto (\tau(u),\eta(u))$ are one-to-one functions.  Suppose that the relationship $x=a(u,\theta)$ in the baseline association \eqref{eq:assoc} can be decomposed as 
\begin{subequations}
\label{eq:decomposition}
\begin{align}
H(x) & = \eta(u), \label{eq:condition} \\
T(x) & = b(\tau(u), \theta). \label{eq:condition0} 
\end{align}
\end{subequations}
This decomposition immediately suggests an alternative association.  Let $(V_T,V_H) \in \VV_T \times \VV_H$ be the image of $U$ under $(\tau,\eta)$, and let $\prob_{V_T|h}$ be the conditional distribution of $V_T$, given $V_H=h$, where $h \in H(\XX)$.  Since $H(x)$ provides no information about $\theta$, we can take a new association 
\begin{equation}
\label{eq:cond-aeqn}
T(X) = b(V_T,\theta), \quad V_T \sim \prob_{V_T | H(x)}. 
\end{equation}
We shall refer to this as a \emph{conditional association}.  This alternative association can be understood via a certain hierarchical representation of the sampling model; see Remark~\ref{re:hierarchy}.  The important point is that $\tau$ can often be chosen so that $V_T$ is of lower dimension than $U$.  In fact, $V_T$ will often have dimension the same as that of $\theta$.  In addition to providing a sort of summary of the data, like in the classical context, this auxiliary variable dimension reduction has a unique advantage in the IM context: efficient predictive random sets for the lower-dimensional $V_T$ are easier to construct.  Furthermore, the conditioning aspect sharpens our predictive ability, improving efficiency even more.  We witnessed, empirically, these gains in efficiency in the simple example in Section~\ref{SS:normal.example}.  Some further remarks on this conditional association, and its connections to Fisher's sufficiency and Bayes theorem, are collected in Section~\ref{SS:remarks}.  

Once a decomposition \eqref{eq:decomposition} is available, construction of the corresponding IM follows exactly as in Section~\ref{S:intro}.  To simplify the presentation later on, here we restate the three-step construction of a \emph{conditional IM}.    

\begin{astep}
Associate $T(x)$ and $\theta$ with the new auxiliary variable $v_T = \tau(u)$ to get the collection of sets $\Theta_{T(x)}(v_T) = \{\theta: T(x) = b(v_T,\theta)\}$, $v_T \in \VV_T$, based on \eqref{eq:cond-aeqn}.  
\end{astep}

\begin{pstep}
Fix $h=H(x)$.  Predict the unobserved value $\vstar_T$ of $V_T$ with a \emph{conditionally admissible} predictive random set $\S \sim \prob_{\S|h}$ (see Section~\ref{SS:condcred}).    
\end{pstep}

\begin{cstep}
Combine results of the A- and P-steps to get 
\begin{equation}
\label{eq:condfocal}
\Theta_{T(x)}(\S) = \bigcup_{v_T \in \S} \Theta_{T(x)}(v_T) \subseteq \Theta. 
\end{equation}
Then the corresponding conditional belief and plausibility functions are given by 
\begin{equation}
\label{eq:condbelief}
\begin{split}
\cbel_{T(x)|h}(A;\S) & = \prob_{\S|h}\{\Theta_{T(x)}(\S) \subseteq A \mid \Theta_{T(x)}(\S) \neq \varnothing\} \\
\cpl_{T(x)|h}(A;\S) & = 1-\cbel_{T(x)|h}(A^c; \S). 
\end{split}
\end{equation}
These functions can be used for inference on $\theta$ just like those in Section~\ref{S:ims}.  The notation $\cbel$ and $\cpl$ is meant to indicate that these are belief and plausibility functions, depending on $(T,H)(x)$, based on predicting the lower-dimensional auxiliary variable $V_T = \tau(U)$ in the conditional association \eqref{eq:cond-aeqn}.   
\end{cstep}

When a decomposition \eqref{eq:decomposition} is available, the conditional association \eqref{eq:cond-aeqn} and the corresponding conditional IM analysis is intuitively reasonable.  One could ask, however, if there is any loss in using IMs built from \eqref{eq:cond-aeqn} instead of \eqref{eq:assoc}.  The following proposition establishes that there is no loss.  

\begin{proposition}
\label{prop:conditioning}
Suppose the baseline association \eqref{eq:assoc} admits a decomposition of the form \eqref{eq:decomposition}.  Let $\S$ be a predictive random set for $U$ in the baseline association with the property that $\prob_\S\{\Theta_x(\S) \neq \varnothing\} > 0$ for all $x$.  Then there exists a predictive random set $\S_{H(x)}$ for $\tau(U)$, depending on $\S$ and $H(x)$, in the conditional association such that $\bel_x(A;\S) = \cbel_{T(x)}(A; \S_{H(x)})$ for all $x$ and all assertions $A \subseteq \Theta$.  
\end{proposition}

See Appendix~\ref{S:proofs} for the proof.  This says that the conditional association is as good a starting point as the baseline association in the sense that any belief function that obtains from the latter can be matched by one that obtains from the former.  Therefore, the best conditional IM can be no worse than the best baseline IM.  However, as in Section~\ref{SS:normal.example}, by working with predictive random sets for the lower-dimensional auxiliary variable in the conditional association, efficiency can be improved.  The point of this paper, in fact, is that the best conditional IM is more efficient than the best baseline IM.  

A shortcoming of Proposition~\ref{prop:conditioning} is that the predictive random set $\S_{H(x)}$ constructed in the proof may not be valid for prediction of $\tau(U)$, which prevents us from making a direct efficiency comparison of conditional versus baseline IMs.  However, in the special case where $\tau(U)$ and $\eta(U)$ are independent, validity of that predictive random set obtains.  

\begin{corollary}
\label{co:indep.cond}
In the setup of Proposition~\ref{prop:conditioning}, suppose that $\tau(U)$ and $\eta(U)$ are independent.  Then the predictive random set $\S'=\tau(\S)$ for $\tau(U)$ is valid and $\bel_x(A;\S) = \cbel_{T(x)}(A; \S')$ for all $x$ and all assertions $A \subseteq \Theta$.  
\end{corollary}

Since the predictive random set for $\tau(U)$ is valid, a requirement for IMs, in this case we can conclude that the conditional IM is at least as efficient as the baseline IM.  The independence condition holds in many examples; see Section~\ref{S:more-examples}.  


\subsection{Remarks}
\label{SS:remarks}

\begin{remark}
\label{re:separable}
More general decompositions in \eqref{eq:decomposition} are possible.  That is, one may replace ``$H(x) = \eta(u)$'' in \eqref{eq:condition} with ``$c(x,u) = 0$'' for a function $c$.  However, this more general ``non-separable'' case may not fit into the context of the conditional validity theorem; see Theorem~\ref{thm:cond.valid}.  We will have more to say about this in Sections~\ref{SS:condcred} and \ref{S:gcim}.  
\end{remark}


\begin{remark}
\label{re:hierarchy}
The decomposition \eqref{eq:decomposition} boils down to a particular hierarchical representation of the sampling model for $X$.  Indeed, for functions $H$ and $T$ as in \eqref{eq:decomposition}, with $V_H=\eta(U)$, and $V_T = \tau(U)$, data $X \sim \prob_{X|\theta}$ can be simulated as follows. 
\begin{enumerate}
\item[1.] Get $(V_T,V_H)$ by sampling $V_H \sim \prob_{V_H}$ and $V_T \mid V_H \sim \prob_{V_T|V_H}$;
\item[2.] Obtain $X$ by solving the system $H(X)=V_H$ and $T(X)=b(V_T,\theta)$.
\end{enumerate}
This hierarchical model representation also provides the following insight: when $X=x$ is observed, so too is the value of $V_H$, and this knowledge can be used to update the auxiliary variable distribution, analogous to Bayes' theorem.   
\end{remark}

\begin{remark}
\label{re:sufficiency}
There are close connections between the conditional IM and Fisher's notion of sufficiency.  At a high level, both theories provide a sort of dimension reduction.  The key difference between the two is that sufficiency focuses on reducing the dimension of the data, while our approach focuses on reducing the dimension of the auxiliary variable.  Although the conditional IM can, in some cases, correspond to a sufficient statistic-type of reduction, this is not necessary; see the remarks at the end of Section~\ref{SS:student}.  Proper conditioning appears to be more important than the use of sufficient statistics.  In fact, in some cases, it is possible, within the IM framework, to reduce the dimension further than that which is provided by sufficiency; see Section~\ref{S:gcim}.  
\end{remark}

\begin{remark}
\label{re:bayes}
As we mentioned previously, conditional IMs have some connections to Bayes' theorem, in particular, in how information is combined or aggregated across samples.  In fact, it can be shown that, in a certain sense, the Bayes solution is a special case of conditional IMs.  To see this, consider a simple but generic example.  The Bayes model, cast in terms of associations, is of the following form:
\[ \theta = U_0, \quad U_0 \sim \prob_{U_0} \quad \text{and} \quad X=a(U_0,U_1), \quad U_1 \sim \prob_{U_1}, \]
where $\prob_U$ for $U=(U_0,U_1)$ is such that $U_1$ is conditionally independent given $U_0$.  Here $\prob_{U_0}$ is like the prior, and the distribution induced by $u_1 \mapsto a(\theta,u_1)$ given $U_0=\theta$ determines the likelihood.  It is clear that the function $a(U_0,U_1)$ is fully observed, so the conditional IM strategy would employ the conditional distribution of $U_0$ given the observed value $x$ of $a(U_0,U_1)$.  It is not hard to see that the belief function based on the ``naive'' predictive random set $\S=\{U_0\}$ is exactly the Bayesian posterior distribution function.  So in any problem with a known prior distribution, the Bayes solution can be obtained as a special case of the conditional IM.  No non-naive predictive random set is needed here because the naive IM itself is valid; this is consistent with the simple corresponding fact for posterior probabilities under a Bayes model with known prior.  
\end{remark}

\begin{remark}
\label{re:more.bayes}
As a follow-up to Remark~\ref{re:bayes}, since a full prior is not required to construct a conditional IM, it is possible to develop an inferential framework based on conditional IMs and ``partial prior information.''  For example, valid prior information may be available for some but not all components of $\theta$.  Incorporating the prior information where it is available while remaining prior-free where it is not can be obtained by slight extension of the argument in the previous remark.  This important application of conditional IMs deserves further investigation.  See, also, \cite{xie.singh.2012}.  
\end{remark}

\subsection{Validity of conditional IMs}
\label{SS:condcred}

Here we extend the validity results in \citet{imbasics} to the conditional IM context.  The main obstacle is that the distribution function $\prob_\S$, determined by the conditional distribution $\prob_{V_T|H(x)}$ in \eqref{eq:decomposition}, depends on data through the value $H(x)$.  This is handled in Theorem~\ref{thm:cond.valid} below by conditioning on the observed value of $H(X)$.     

Fix $h \in H(\XX)$, and let $\SS_h$ be a collection of closed $\prob_{V_T|h}$-measurable subsets of $\VV_T$ that contains both $\varnothing$ and $\VV_T$.  Like before, we also assume that $\SS_h$ is nested in the sense that either $S \subseteq S'$ or $S' \subseteq S$ for all $S,S' \in \SS_h$.  Then $\S$ is a \emph{conditionally admissible} predictive random set, given $h$, if its distribution $\prob_{\S|h}$ satisfies 
\begin{equation}
\label{eq:cond.natural.measure}
\prob_{\S|h}\{\S \subseteq K\} = \sup_{S \in \SS_h: S \subseteq K} \prob_{V_T|h}\{S\}, \quad K \subseteq \VV_T. 
\end{equation}
In this case, the distribution of $\S$ depends on the particular $h$.  We now have the following extension of the validity theorem to the case of conditional IMs.  

\begin{theorem}
\label{thm:cond.valid}
For any $h$, suppose that $\S$ is conditionally admissible, given $h$, with distribution $\prob_{\S|h}$ as in \eqref{eq:cond.natural.measure}.  If $\Theta_{T(x)}(\S) \neq \varnothing$ with $\prob_{\S|h}$-probability~1 for all $x$ such that $H(x)=h$, then the conditional IM is conditionally valid, i.e., for any $A \subseteq \Theta$, 
\begin{equation}
\label{eq:cond.valid.bel}
\sup_{\theta \not\in A} \prob_{X|\theta}\{\cbel_{T(X)|h}(A;\S) \geq 1-\alpha \mid H(X) = h\} \leq \alpha, \quad \forall \; \alpha \in (0,1). 
\end{equation}
\end{theorem}

Now is a good time to recall Remark~\ref{re:separable}.  More general decompositions of the baseline association are allowed in the discussion in Section~\ref{SS:cim}, but only for the ``separable'' version \eqref{eq:condition} is it possible to prove a conditional validity theorem.  The point is that a condition like $c(X,U)=0$ does not identify a fixed subset of the sample space on which probability calculations can be restricted---the subspace would depend on $U$.  

Since the calibration property in Theorem~\ref{thm:cond.valid} holds for all assertions $A$, we may translate \eqref{eq:cond.valid.bel} to a statement in terms of the corresponding plausibility function:
\begin{equation}
\label{eq:cond.valid.pl}
\sup_{\theta \in A} \prob_{X|\theta}\{\cpl_{T(X)|h}(A;\S) \leq \alpha \mid H(X)=h\} \leq \alpha, \quad \forall\;\alpha \in (0,1). 
\end{equation}
So, in addition to providing an objective scale for interpreting the conditional belief and plausibility function values, \eqref{eq:cond.valid.pl} provides desirable properties of conditional IM-based frequentist procedures.  For example, if $h=H(x)$ is observed, the conditional $100(1-\alpha)\%$ plausibility region for $\theta$ is $\{\theta: \cpl_{T(x)|h}(\theta; \S) > \alpha\}$.  Then, by \eqref{eq:cond.valid.pl}, the conditional coverage probability is $\prob_{X|\theta}\{\cpl_{T(X)|h}(\theta;\S) > \alpha \mid H(X)=h\} \geq 1-\alpha$.  In Fisher's mind, this is a more meaningful coverage probability since it is conditioned on a particular aspect of the observed data, namely, $H(x)=h$.  In other words, the probability calculation focuses on a relevant subset $\{x: H(x)=h\}$ of the sample space.  In some cases, though, conditional validity is the same as ordinary validity.  

\begin{corollary}
\label{co:cond.valid}
Suppose that the predictive random set $\S$ does not depend on the observed $H(x)=h$, so that $\prob_{\S|h} \equiv \prob_\S$ and $\cbel_{T(x)|h} \equiv \cbel_{T(x)}$.  Then under the conditions of Theorem~\ref{thm:cond.valid}, the conditional IM is unconditionally valid, i.e., for any $A \subseteq \Theta$, 
\[ \sup_{\theta \not\in A} \prob_{X|\theta}\{\cbel_{T(X)}(A;\S) \geq 1-\alpha\} \leq \alpha, \quad \forall \; \alpha \in (0,1).  \]
\end{corollary}

Two possible ways the condition of Corollary~\ref{co:cond.valid} may hold are as follows.  First, in the P-step, the user may specify $\S$ directly without dependence on the observed $H(x)=h$; see Section~\ref{SS:student}.  Second, it could happen that $V_T$ and $V_H$ are statistically independent, in which case the distribution $\prob_\S$ for $\S$ is determined by the marginal distribution of $V_T$, which does not depend on $h$.

\section{Finding conditional associations}
\label{S:finding}

\subsection{Familiar things: likelihood and symmetry}

In many problems, finding a decomposition \eqref{eq:decomposition} and the corresponding conditional association is easy to do.  In general, the definition of sufficiency implies that we can define a conditional association via, say, the marginal distribution of the minimal sufficient statistic; see Section~\ref{SS:gamma}.  In standard problems, such as full-rank exponential families, minimal sufficient statistics are easily obtained, so this is probably the simplest approach.  This, of course, includes both discrete and continuous problems.  Similarly, if the problem has a group structure, invariance considerations can be used to find a decomposition; see Section~\ref{SS:student}.  But one can consider other conditional associations if desirable.  For example, when the minimal sufficient statistic has dimension larger than that of the parameter, like in curved exponential families, then some special conditioning can potentially further reduce the dimension; see Section~\ref{SS:nile}.  


\subsection{A new differential equations-based technique}
\label{SS:diffeq}

Here we describe a novel technique for finding conditional associations, based on differential equations.  The method can be used for going directly from the baseline association to something lower-dimensional.  In fact, in those nice problems mentioned above, it is easy to check that this differential equation-based technique reproduces the solutions based on minimal sufficiency, group invariance, etc.  However, in our experience, this new approach is especially powerful in cases where the familiar things fail to give a fully satisfactory reduction.  In such cases, the differential equation-based technique can provide a further dimension reduction, beyond what sufficiency alone can give.  

For concreteness, suppose $\Theta \subseteq \RR$; the multi-parameter case can be handled similarly, as in Section~\ref{SS:variance.components}.  The intuition is that $\tau$ should map $\UU \subseteq \RR^n$ to $\Theta$, so that $V_T=\tau(U)$ is one-dimensional, like $\theta$.  Moreover, $\eta$ should map $\UU$ into a $(n-1)$-dimensional manifold in $\RR^n$, and be insensitive to changes in $\theta$ in the following sense.  For baseline association $x=a(\theta,u)$, suppose that $u_{x,\theta}$ is the unique solution for $u$.  Then for fixed $x$, we require that $\eta(u_{x,\theta})$ be constant in $\theta$.  In other words, we require that $\del u_{x,\theta}/\del\theta$ exists and 
\begin{equation}
\label{eq:diffeq}
\underset{n \times 1}{0} = \frac{\del \eta(u_{x,\theta})}{\del \theta} = \underset{\text{$n \times n$, rank $n-1$}}{\frac{\del \eta(u)}{\del u} \Bigr|_{u=u_{x,\theta}}} \cdot \underset{n \times 1}{\frac{\del u_{x,\theta}}{\del \theta}}.  
\end{equation}
It is clear from the construction that, if a solution $\eta$ of this partial differential equation exists, then the value of $\eta(U)$ is fully observed, i.e., there is a corresponding function $H$, not depending on $\theta$, such that $H(X)=\eta(U)$.  So, with appropriate choice of $\tau$, the solution $\eta$ of \eqref{eq:diffeq} determines the decomposition \eqref{eq:decomposition}.  A different but related use of $\theta$-derivatives of the association is presented in \citet{fraser.fraser.staicu.2010}.  

Formal theory on existence of solutions and on solving the differential equation system \eqref{eq:diffeq} is available.  For example, the method of characteristics described in \citet{polyanin2002} is powerful tool for solving such systems.  However, such formalities here will take us too far off track.  Examples of this method in action are given in Section~\ref{SS:nile}, \ref{SS:bvn2}, and \ref{SS:variance.components}.  In all three cases, this differential equations method is applied after an initial step based on sufficiency provides an unsatisfactory dimension reduction.

\section{Three detailed examples}
\label{S:more-examples}

\subsection{A Student-t location problem}
\label{SS:student}

Suppose $X_1,\ldots,X_n$ is an independent sample from a Student-t distribution $\stt_\nu(\theta)$, where the degrees of freedom $\nu$ is known but the location $\theta$ is unknown.  This is a somewhat peculiar problem because there is no satisfactory reduction via sufficiency.  For the IM approach, start with a baseline association $X = \theta 1_n + U$, with $U=(U_1,\ldots,U_n)^\top$ and $U_i \sim \stt_\nu$, independent, for $i=1,\ldots,n$.  For this location parameter problem, invariance considerations suggest the following decomposition:
\[ X-T(X)1_n = U-T(U)1_n \quad \text{and} \quad T(X) = \theta + T(U), \]
where $T(\cdot)$ is the maximum likelihood estimator.  Let $V_T=T(U)$ and $V_H = H(U)=U-T(U)1_n$.  If $h$ is the observed $H(X)$, then it follows from the result of \citet{bn1983} that the conditional distribution of $V_T$, given $V_H=h$, has a density
\[ f_{\nu,h}(v_T) = c(\nu,h) \prod_{i=1}^n \bigl\{ \nu + (v_T+h_i)^2 \bigr\}^{-(\nu+1)/2}, \]
where $c(\nu,h)$ is a normalizing constant that depends only on $\nu$ and $h$.  If we write $F_{\nu,h}$ for the distribution function corresponding to the density $f_{\nu,h}$ above, then a conditional IM for $\theta$ can be built based on the following association: 
\[ T(X) = \theta + F_{\nu,h}^{-1}(W), \quad W \sim \unif(0,1). \]
With this conditional association, we are ready for the P- and C-steps.  For simplicity, in the P-step we elect to take the predictive random set $\S$ as in \eqref{eq:default.prs}; this also has some theoretical justification since $f_{\nu,h}$ should be approximately symmetric about $v_T=0$ \citep[][Sec.~4.3.2]{imbasics}.  For the C-step, the random set $\Theta_{T(x)}(\S)$ is 
\[ \bigl[ T(x) - F_{\nu,h}^{-1}\bigl(\tfrac12+|W-\tfrac12|\bigr), \, T(x) - F_{\nu,h}^{-1}\bigl(\tfrac12-|W-\tfrac12|\bigr) \bigl], \quad W \sim \unif(0,1). \]
From this point, numerical methods can be used to compute the conditional belief and plausibility functions.  For example, if $A=\{\theta\}$ is a singleton assertion, then 
\[ \cpl_{T(x)|h}(\theta; \S) = 1-\bigl|1-2F_{\nu,h}(\theta-T(x)) \bigr|, \]
and the corresponding $100(1-\alpha)$\% plausibility interval for $\theta$ is 
\[ \{\theta: \cpl_{T(x)|h}(\theta; \S) > \alpha\} = \bigl( T(x) + F_{\nu,h}^{-1}(\alpha/2), \, T(x) + F_{\nu,h}^{-1}(1-\alpha/2) \bigr). \]

For illustration, we present the results of a simple simulation study.  In particular, for several pairs $(n,\nu)$, 5000 Monte Carlo samples of size $n$ are obtained from a Student-t distribution with $\nu$ degrees of freedom and center $\theta = 0$.  For each sample, the 95\% plausibility interval for $\theta$ based on the conditional IM above is obtained.  For comparison, we also compute the 95\% confidence interval based on the asymptotic normality of the maximum likelihood estimate, and a 95\% flat-prior Bayesian credible interval.  The results of this simulation are summarized in Table~\ref{table:t}.  We find that the results here are almost indistinguishable, so favor must go to the plausibility intervals, since these have guaranteed coverage for all $n$, while the other two are only asymptotically correct.

\begin{table}
\caption{\label{table:t} Coverage probabilities and expected lengths of the 95\% intervals for $\theta$ in the Student-t example based on, respectively, the conditional IM (CIM), asymptotic normality of the maximum likelihood estimate (MLE), and flat-prior Bayes.}
\centering
\fbox{
\begin{tabular}{cccccccccccc}
& & & \multicolumn{4}{c}{Coverage probability} & & \multicolumn{4}{c}{Expected length} \\
& & & \multicolumn{4}{c}{$\nu$} & & \multicolumn{4}{c}{$\nu$} \\
\cline{4-7} \cline{9-12} 
Method & $n$ & & 3 & 5 & 10 & 25 & & 3 & 5 & 10 & 25 \\
\hline
CIM & 5 & & 0.944 & 0.949 & 0.951 & 0.949 & & 2.28 & 2.08 & 1.93 & 1.83 \\
& 10 & & 0.949 & 0.951 & 0.952 & 0.953 & & 1.56 & 1.45 & 1.35 & 1.29 \\
& 25 & & 0.953 & 0.944 & 0.951 & 0.949 & & 0.97 & 0.91 & 0.85 & 0.81 \\
& 50 & & 0.953 & 0.951 & 0.953 & 0.947 & & 0.68 & 0.64 & 0.60 & 0.58 \\
\hline
MLE & 5 & & 0.931 & 0.939 & 0.940 & 0.946 & & 2.10 & 1.99 & 1.88 & 1.80 \\
& 10 & & 0.953 & 0.942 & 0.949 & 0.941 & & 1.51 & 1.42 & 1.334 & 1.28 \\
& 25 & & 0.938 & 0.948 & 0.947 & 0.950 & & 0.96 & 0.90 & 0.85 & 0.81 \\
& 50 & & 0.946 & 0.946 & 0.954 & 0.956 & & 0.68 & 0.64 & 0.60 & 0.57 \\
\hline
Bayes & 5 & & 0.949 & 0.955 & 0.946 & 0.948 & & 2.28 & 2.08 & 1.93 & 1.82 \\
& 10 & & 0.960 & 0.948 & 0.951 & 0.942 & & 1.56 & 1.45 & 1.35 & 1.29 \\
& 25 & & 0.943 & 0.949 & 0.948 & 0.950 & & 0.97 & 0.91 & 0.85 & 0.81 \\
& 50 & & 0.947 & 0.947 & 0.955 & 0.956 & & 0.68 & 0.64 & 0.60 & 0.58 \\
\end{tabular}
}
\end{table}

We also did the conditional IM calculations with an alternative decomposition, which took $V_T=U_1$ and $V_H=(0,U_2-U_1,\ldots,U_n-U_1)$.  We were surprised to see that the results obtained with this ``naive'' decomposition were indistinguishable from those shown here based on the arguably more reasonable maximum likelihood-driven decomposition.  This suggests that the particular choice of decomposition may not be so important; instead, it is the conditioning part that seems to matter most; see \citet{fraser2004}.

\subsection{Fisher's problem of the Nile}
\label{SS:nile}

Suppose two independent exponential samples, namely $X_1 = (X_{11},\ldots,X_{1n})$ and $X_2 = (X_{21},\ldots,X_{2n})$, are available, the first with mean $\theta^{-1}$ and the second with mean $\theta$.  The goal is to make inference on $\theta > 0$.  The name comes from an application \citep{fisher1973} to fertility of land in the Nile river valley.  In this example, the maximum likelihood estimate is not sufficient, so conditioning on an ancillary statistic is recommended. 

Sufficiency considerations suggest the following initial dimension reduction step: 
\[ S(X_1) = \theta^{-1} U_1 \quad \text{and} \quad S(X_2) = \theta U_2, \quad U_1,U_2 \sim \gam(n,1), \]
where $S(X_i) = \sum_{j=1}^n X_{ij}$.  But efficiency can be gained by considering a further reduction to a scalar auxiliary variable.  Here we employ the differential equation technique in Section~\ref{SS:diffeq}.  Start with $u_{x,\theta} = (\theta S(x_1), \theta^{-1}S(x_2))^\top$.  Differentiating with respect to $\theta$ reveals that our (real valued) conditioning function $\eta$ must satisfy 
\[ \frac{\del\eta(u)}{\del u} \Bigr|_{u=u_{x,\theta}} \binom{S(x_1)}{-\theta^{-2}S(x_2)} = 0. \]
If we take $\eta(u) = \{u_1u_2\}^{1/2}$, then 
\[ \frac{\del\eta(u)}{\del u} \Bigr|_{u=u_{x,\theta}} = \frac{1}{2\{S(x_1)S(x_2)\}^{1/2}} \bigl(\theta^{-1} S(x_2), \,\theta S(x_1)\bigr) \]
and, clearly, this satisfies the differential equation above.  Therefore, for \eqref{eq:decomposition}, we take 
\begin{equation}
\label{eq:nile3}
H(X) = V_H \quad \text{and} \quad T(X) = \theta V_T, 
\end{equation}
where $T(X) = \{S(X_1)/S(X_2) \}^{1/2}$, $H(X) = \{S(X_1) S(X_2)\}^{1/2}$, $V_T = \{U_1/U_2\}^{1/2}$, and $V_H = \{U_1U_2\}^{1/2}$.  These quantities are familiar from the classical approach: $T(X)$ is the maximum likelihood estimate of $\theta$, $H(X)$ is an ancillary statistic, and the pair $(T,H)(X)$ is a jointly minimal sufficient statistic \citep{ghoshreidfraser2010}.  

By \eqref{eq:nile3} and our general discussion in Section~\ref{SS:cim}, we can focus on a conditional association based on $T(X)=\theta V_T$.  The conditional distribution of $V_T$ given $V_H=h$ is a generalized inverse Gaussian distribution \citep{bn1977} with density function 
\begin{equation}
\label{eq:gig}
f_h(v_T) = \frac{1}{2 v_T K_0(2h)} \exp\{-h (v_T^{-1}+v_T) \}, 
\end{equation}
where $K_0$ is the modified Bessel function of the second kind.  As a final simplifying step, write the conditional association as 
\begin{equation}
\label{eq:nile4}
T(X) = \theta F_h^{-1}(W), \quad W \sim {\sf Unif}(0,1), 
\end{equation}
where $F_h$ is the distribution function corresponding to the density $f_h$ in \eqref{eq:gig}.  This completes the A-step.  If we take $\S$ as in \eqref{eq:default.prs} for the P-step, then the C-step gives
\[ \Theta_{T(x)}(\S) = \Bigl[ \frac{T(x)}{F_h^{-1}(\tfrac12 + |W-\tfrac12|)}, \, \frac{T(x)}{F_h^{-1}(\tfrac12-|W-\tfrac12|)} \Bigr], \quad W \sim \unif(0,1). \]
From this, the conditional belief/plausibility functions are readily evaluated. 

For illustration, we display plausibility functions $\cpl_t(\theta;\S)$ for two conditional IMs.  The first is based on that derived above; the second is based on a similar derivation, but we ignore $V_H$ and simply work with the marginal distribution of $V_T$ in \eqref{eq:nile3}.  Figure~\ref{fig:nile} shows plausibility functions for $T(x)=0.90$, with $n=20$ and true $\theta=1$, sampled from its conditional distribution given $h$, for two different values of $h$.  In this case, if $h$ is large (i.e., $h > n$), then the bona fide conditional IM has narrower level sets than the naive conditional IM.  The opposite is true when $h$ is small (i.e., $h < n$).  This is due to the fact that the conditional Fisher information in $T$ is an increasing function in $h$; see \citet[][Example~1]{ghoshreidfraser2010}.  Therefore, $T$ has more variability when $h$ is small, and this adjustment should be reflected in the plausibility function.  The bona fide conditional IM catches this phenomenon while the naive one does not.

\begin{figure}
\centering
\fbox{
\subfigure[$h=15$]{\scalebox{0.5}{\includegraphics{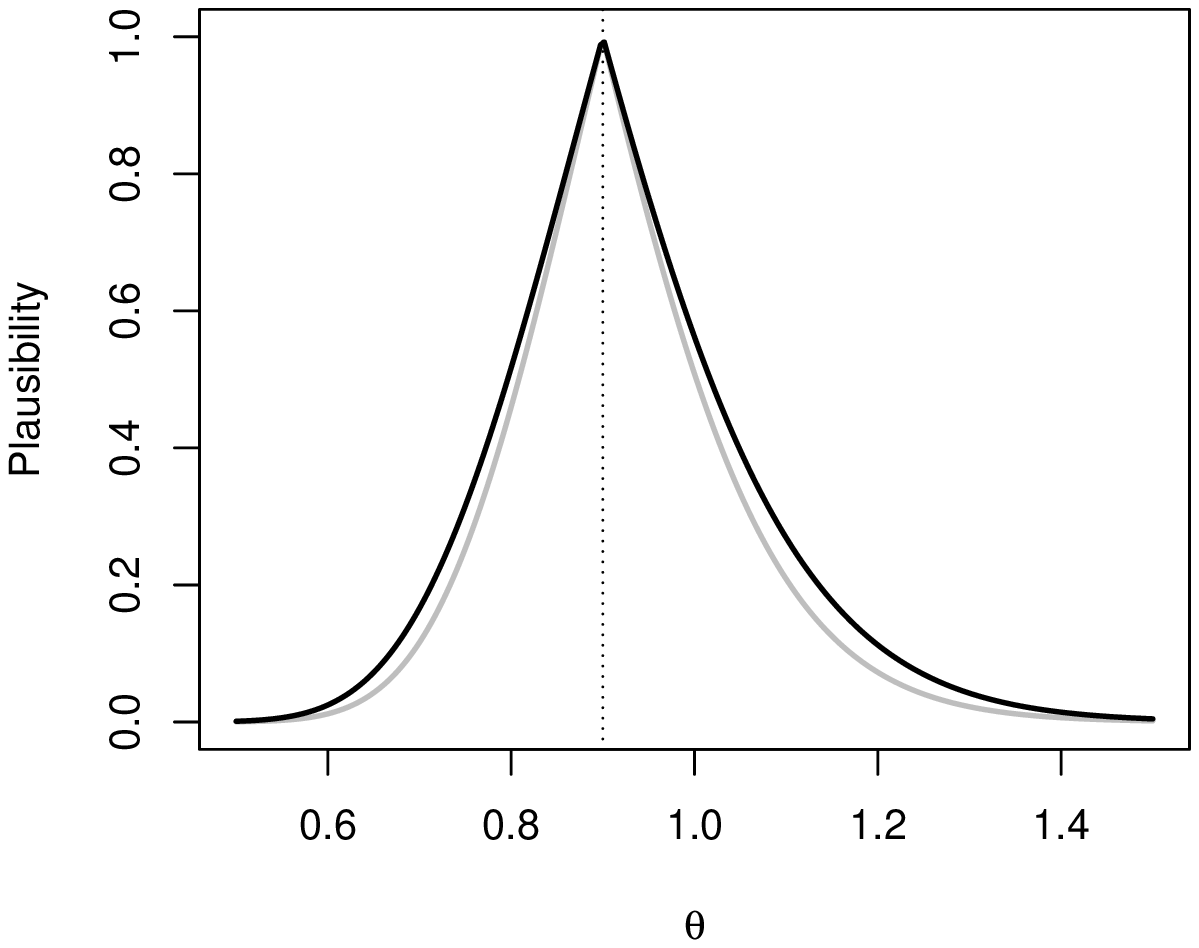}}}
\subfigure[$h=25$]{\scalebox{0.5}{\includegraphics{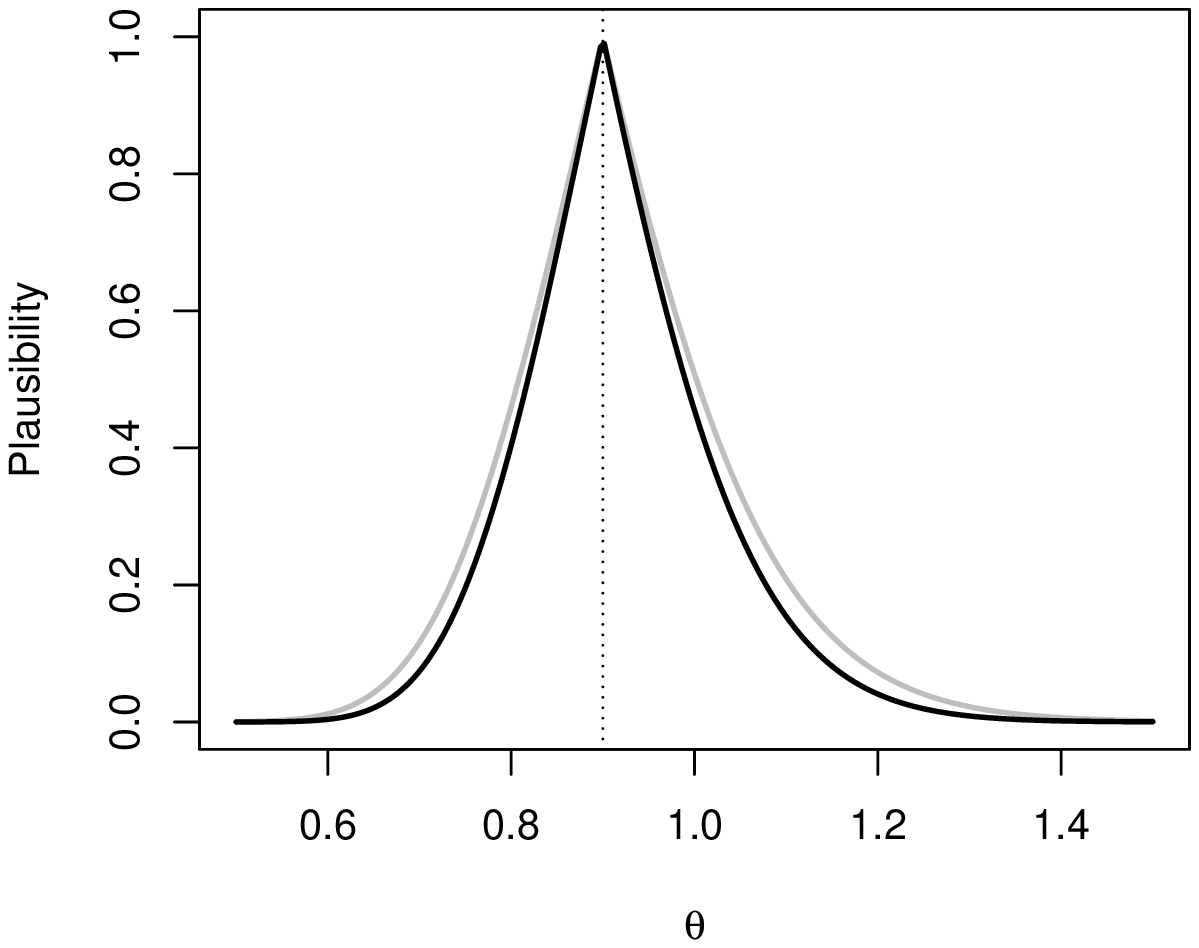}}}
}
\caption{\label{fig:nile} Plausibility functions for the conditional IM (black) and the ``naive'' conditional IM (gray) in the Nile example, with $T=0.90$, $n=20$, and the true $\theta=1$.  Gray curves in the two plots are the same since the naive conditional IM does not depend on $h$.}
\end{figure}

\subsection{A two-parameter gamma problem}
\label{SS:gamma}

Let $X=(X_1,\ldots,X_n)$ be an independent sample from $\gam(\theta_1,\theta_2)$, where $\theta_1 > 0$ and $\theta_2 > 0$ are the shape and scale parameters, respectively, both unknown.  In this case, we may construct a conditional association based on the marginal distribution of the two-dimensional complete sufficient statistic, which we choose to represent as $T_1 = \sum_{i=1}^n X_i$ and $T_2 = n^{-1}\sum_{i=1}^n \log X_i - \log(T_1/n)$.  Then we have a conditional association
\[ T_1 = \theta_2 F_{n\theta_1}^{-1}(U_1) \quad \text{and} \quad T_2 = G_{\theta_1}^{-1}(U_2), \]
where $U_1,U_2$ are independent $\unif(0,1)$, $F_a$ is the distribution function of $\gam(a,1)$, and $G_b$ is some distribution function without a familiar form.  For the P-step, consider an analogue of the default predictive random set \eqref{eq:default.prs}, given by the random square: 
\[ \S = \{(u_1,u_2): \max(|u_1-\tfrac12|,|u_2-\tfrac12|) \leq \max(|U_1-\tfrac12|,|U_2-\tfrac12|)\}, \]
with $U_1,U_2 \iid \unif(0,1)$.  In this case, with observed $t=(t_1,t_2)$, the C-step gives 
\[ \Theta_t(\S) = \{(\theta_1,\theta_2): \max(|F_{n\theta_1}(t_1/\theta_2) - \tfrac12|,|G_{\theta_1}(t_2)-\tfrac12|) \leq \max(|U_1-\tfrac12|,|U_2-\tfrac12|)\}.  \]
From here, we can write down the plausibility function for a singleton assertion:
\[ \cpl_t(\{\theta_1,\theta_2\}; \S) = 1- \max\bigl\{ |2F_{n\theta_1}(t_1/\theta_2) - 1|, \, |2G_{\theta_1}(t_2) - 1| \bigr\}^2. \]
Evaluating $G_{\theta_1}(\cdot)$ requires Monte Carlo but, since $T_2$ is $\theta_2$-ancillary, the same Monte Carlo samples can be used for all candidate $\theta_2$ values.  

For illustration, we simulated a single sample of size $n=25$ from a gamma distribution with shape $\theta_1=7$ and scale $\theta_2=3$.  Figure~\ref{fig:gamma.contour} displays a sample of size 5000 from a Bayesian posterior distribution for $(\theta_1,\theta_2)$ based on Jeffreys' prior.  Also displayed are the 90\% confidence ellipse based on the asymptotic normality of the maximum likelihood estimator, and the 90\% conditional IM plausibility region 
\[ \{(\theta_1,\theta_2): \cpl_t(\{\theta_1,\theta_2\}; \S) > 0.1\}. \]
Besides having guaranteed coverage, the plausibility region captures the non-elliptical shape of the posterior.  For larger $n$, all three regions will have a similar shape.

\begin{figure}
\centering
\fbox{
\scalebox{0.5}{\includegraphics{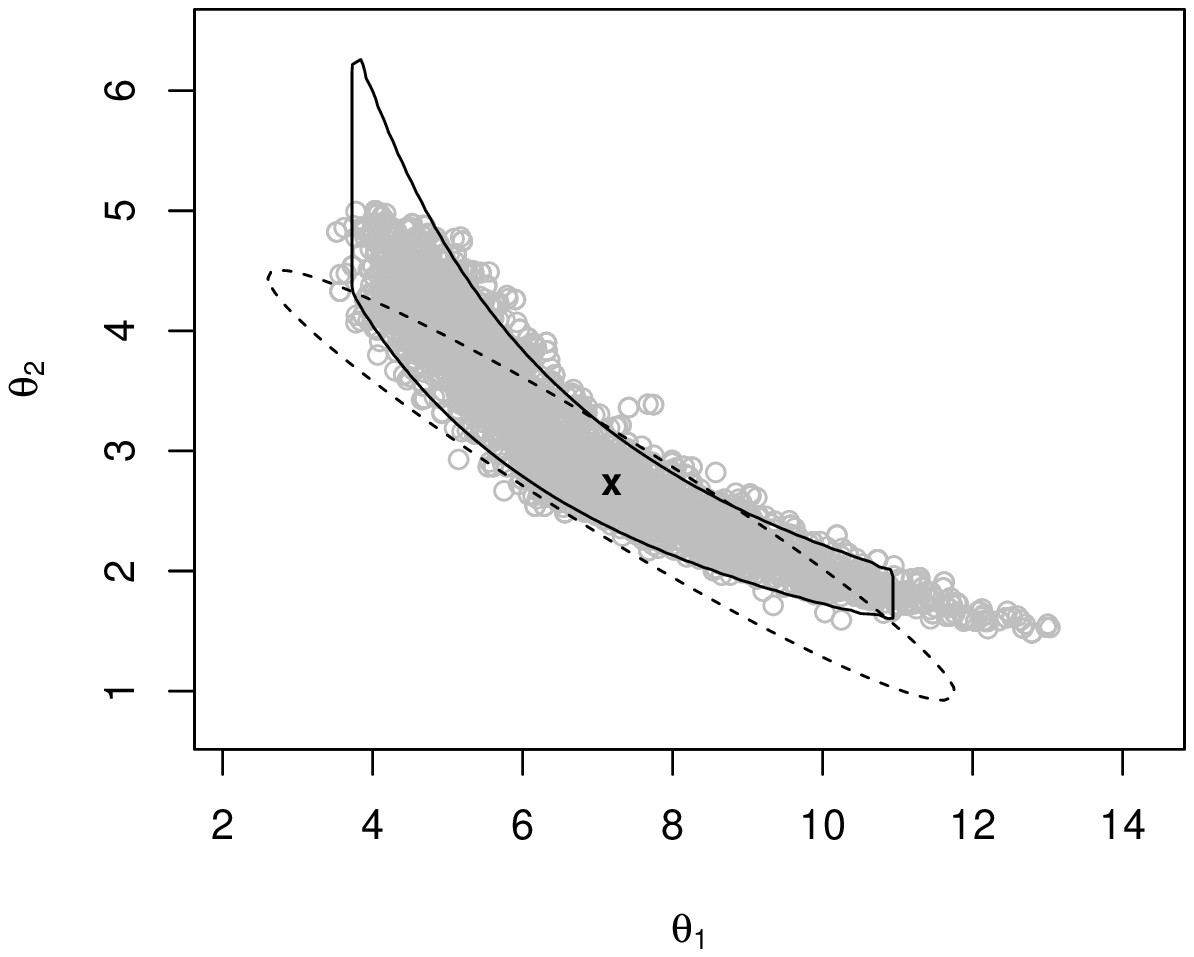}}
}
\caption{\label{fig:gamma.contour} Bayesian posterior sample (gray) based on Jeffreys' prior, the 90\% confidence ellipse based on asymptotic normality of the maximum likelihood estimator (dashed), and the 90\% conditional IM plausibility region.}
\end{figure}

\section{Local conditional IMs}
\label{S:gcim}

\subsection{Motivation: bivariate normal model}
\label{SS:bvn1}

So far we have seen that the conditional IM approach is successful in problems where the baseline association admits a decomposition of the form \eqref{eq:decomposition}.  However, as alluded to above, there are interesting and important problems where apparently no such decomposition exists.  Next is one such problem, which may be considered as a ``benchmark example'' for conditional inference \citep[][Example~5]{ghoshreidfraser2010}.  

Suppose $(X_{11},X_{21}),\ldots,(X_{1n},X_{2n})$ is an independent sample from a standard bivariate normal distribution with zero means, unit variances, but unknown correlation coefficient $\theta \in (-1,1)$.  A natural first step towards inference on $\theta$ is to take advantage of the fact that $X_1+X_2$ and $X_1-X_2$ are independent.  In particular, by defining 
\[ X_1 \gets \frac12 \sum_{i=1}^n (X_{1i}+X_{2i})^2 \quad \text{and} \quad X_2 \gets \frac12 \sum_{i=1}^n (X_{1i}-X_{2i})^2, \]
we may rewrite the baseline association as 
\begin{equation}
\label{eq:bvn.amodel.1}
X_1 = (1+\theta)U_1 \quad \text{and} \quad X_2 = (1-\theta)U_2, \quad U_1,U_2 \sim \chisq(n).
\end{equation}
Sufficiency justifies this first reduction.  Equation \eqref{eq:bvn.amodel.1} is equivalent to
\begin{equation}
\label{eq:bvn.amodel.2}
\frac{X_1}{U_1} + \frac{X_2}{U_2} = 2 \quad \text{and} \quad \frac{X_1}{X_2} = \frac{1+\theta}{1-\theta} \, \frac{U_1}{U_2} . 
\end{equation}
The first equation depends on data and auxiliary variable---free of $\theta$---while the second depends also on $\theta$.  But note that the first expression in \eqref{eq:bvn.amodel.2} is not of the form specified in \eqref{eq:condition}.  Actually, this first expression is of the more general ``non-separable'' form $c(X,U) = 0$ described in Remark~\ref{re:separable}.  So, although \eqref{eq:bvn.amodel.2} provides a suitable decomposition of the baseline association, the requirements of Theorem~\ref{thm:cond.valid} are not met, so the resulting conditional IM may not be valid.  This warrants an alternative approach.  

To elaborate on this last point, observe that the distribution for $\theta$ obtained via the distribution of $(U_1,U_2)$, given $X_1/U_1 + X_2/U_2 = 2$, is exactly a type of fiducial distribution.  As we pointed out in Section~\ref{SS:notation}, conditioning on the full data, $(X_1,X_2)$ in this case, for fixed $\theta$, makes the distribution of $(U_1,U_2)$ degenerate.  Therefore, ``continuing to regard'' $(U_1,U_2)$ as independent chi-squares, given data, may be difficult to justify.

\subsection{Relaxing \eqref{eq:condition} via localization}
\label{SS:localization}

As describe above, the separability in \eqref{eq:condition} can be too strict, but extending the conditional validity theorem to allow non-separablility appears difficult.  The idea here is to relax \eqref{eq:condition} in a different direction.  Specifically, we propose to allow the pair of function $(H,\eta)$ in \eqref{eq:condition} to depend, locally, on the parameter.  This generalization allows us some additional flexibility in finding an auxiliary variable dimension reduction.  

Start by fixing an arbitrary $\theta_0 \in \Theta$.  As in Section~\ref{SS:cim}, consider a pair of functions $(T,H_{\theta_0})$, depending on $\theta_0$, such that $x \mapsto (T(x),H_{\theta_0}(x))$ is one-to-one.  Now take the corresponding functions $u \mapsto (\tau(u), \eta_{\theta_0}(u))$, one-to-one, such that the baseline association, at $\theta=\theta_0$, can be decomposed as 
\begin{equation}
\label{eq:gcim.decomp}
H_{\theta_0}(X) = \eta_{\theta_0}(U) \quad \text{and} \quad T(X) = b(\tau(U), \theta_0). 
\end{equation}
That is, \eqref{eq:gcim.decomp}, with $U \sim \prob_U$, describes the sampling distribution $X \sim \prob_{X|\theta_0}$.  Suppose $H_{\theta_0}(X)=h_0$ is observed.  We can compute the conditional distribution $\prob_{V_T|h_0,\theta_0}$ of $V_T = \tau(U)$ given $\eta_{\theta_0}(U) = h_0$, which is then used to construct predictive random sets.  

From this point, we may proceed exactly as before.  That is, for the A-step, we get sets $\Theta_{T(x)}(v_T) = \{\theta: T(x)=b(v_T,\theta)\}$ just like before.  For the P-step, we pick a conditionally admissible predictive random set $\S \sim \prob_{S|h_0,\theta_0}$.  Finally, the C-step produces conditional the plausibility function 
\[ \cpl_{T(x)|h_0,\theta_0}(A;\S) = 1-\prob_{\S|h_0,\theta_0}\{\Theta_{T(x)}(\S) \subseteq A^c\}, \quad A \subseteq \Theta. \]
We shall refer to the corresponding conditional IM as a \emph{local} conditional IM at $\theta=\theta_0$.  The adjective ``local'' is meant to indicate the dependence of the construction on the particular point $\theta_0$.  As we see below, the validity properties of this local conditional IM are, in a certain sense, also local.

\subsection{Validity of local conditional IMs}
\label{SS:local.valid}

The following theorem shows that for each $\theta_0$ value, the local conditional IM at $\theta_0$ is valid for some important assertions depending on the particular $\theta_0$.  The proof is exactly like that of Theorem~\ref{thm:cond.valid} and, hence, omitted.  

\begin{theorem}
\label{thm:local.valid}
For any $\theta_0$, take $h_0 \in H_{\theta_0}(\XX)$.  Suppose that $\S \sim \prob_{\S|h_0,\theta_0}$ is conditionally admissible.  If $\Theta_{T(x)}(\S) \neq \varnothing$ with $\prob_{\S|h_0,\theta_0}$-probability~1 for all $x$ such that $H_{\theta_0}(x)=h_0$, then the local conditional IM at $\theta_0$ is conditionally valid for $A=\{\theta_0\}$, i.e., 
\[ \prob_{X|\theta_0}\{\cpl_{T(X)|h_0,\theta_0}(\theta_0;\S) \leq \alpha \mid H_{\theta_0}(X) = h_0\} \leq \alpha, \quad \forall \; \alpha \in (0,1). \]
\end{theorem}

The validity result here is not as strong as in Theorem~\ref{thm:cond.valid}, a consequence of the localization.  It does, however, imply that the local conditional plausibility region, 
\begin{equation}
\label{eq:cond.pl.region}
\{\theta: \cpl_{T(x)|H_\theta(x)}(\theta;\S) > \alpha\}, 
\end{equation}
has the nominal (conditional) $1-\alpha$ coverage probability.  This theoretical result is confirmed by the simulation experiment in Section~\ref{SS:bvn2} below.  Observe that, in the definition of conditional plausibility region \eqref{eq:cond.pl.region}, the plausibility function depends on $\theta$ in two places---in the argument (the assertion) and in the local conditional IM itself.  The latter structural dependence of the IM on the particular assertion is consistent with the optimality developments described in \citet{imbasics}.

\subsection{Bivariate normal model, revisited}
\label{SS:bvn2}

Here we demonstrate that the localization technique can be successfully used to solve the bivariate normal problem described above.  Start with the relation in \eqref{eq:bvn.amodel.1}.  Fix $\theta_0$.  To construct the functions $(H, \eta_{\theta_0})$, depending on $\theta_0$, and the corresponding local conditional IM at $\theta_0$, we shall modify the differential equation approach in Section~\ref{SS:diffeq}.  

In this case, if we let $u_{x,\theta} = (x_1/(1+\theta), x_2/(1-\theta))^\top$, then we have 
\[ \frac{\del u_{x,\theta}}{\del\theta} = \Bigl(-\frac{x_1}{(1+\theta)^2}, \frac{x_2}{(1-\theta)^2}\Bigr)^\top. \]
For a local conditional IM at $\theta_0$, we propose to choose a real-valued $\eta_{\theta_0}(u)$ such that $\del \eta_{\theta_0}(u_{x,\theta})$ vanishes at $\theta=\theta_0$.  If we take 
\begin{equation}
\label{eq:phi}
\eta_{\theta_0}(u) = (1+\theta_0)\log u_1 + (1-\theta_0)\log u_2, 
\end{equation}
then 
\[ \frac{\del\eta_{\theta_0}(u)}{\del u} = \Bigl(\frac{1+\theta_0}{u_1}, \,\frac{1-\theta_0}{u_2}\Bigr), \]
so the derivative of $\psi_{H_0}(u_{x,\theta})$ with respect to $\theta$ is 
\begin{align*}
\frac{\del\eta_{\theta_0}(u_{x,\theta})}{\del\theta} & = \frac{\del\eta_{\theta_0}(u)}{\del u} \Bigr|_{u=u_{x,\theta}} \cdot \frac{\del u_{x,\theta}}{\del\theta} \\
& = -\frac{(1+\theta_0)^2}{x_1} \cdot \frac{x_1}{(1+\theta)^2} + \frac{(1-\theta_0)^2}{x_2} \cdot \frac{x_2}{(1-\theta)^2} \\
& = -\frac{(1+\theta_0)^2}{(1+\theta)^2} + \frac{(1-\theta_0)^2}{(1-\theta)^2}.  
\end{align*}
The latter expression clearly evaluates to zero at $\theta=\theta_0$, so $\eta_{\theta_0}$ satisfies the desired differential equation.  The corresponding function $H(x)=H_{\theta_0}(x)$ is given by 
\[ H_{\theta_0}(x) = (1+\theta_0) \log\{x_1/(1+\theta_0)\} + (1-\theta_0)\log\{x_2 / (1-\theta_0)\}. \]
For the local conditional association---the second expression in \eqref{eq:gcim.decomp}---we take 
\[ T(X) = z(\theta) + V_T, \]
where $T(x) = \log(x_1/x_2)$, $z(\theta) = \log\{(1+\theta)/(1-\theta)\}$, and $V_T=T(U)$.  Then $\prob_{V_T|\theta_0,h_0}$ is the conditional distribution of $V_T$, given $(\theta_0,h_0)$, where $h_0$ is the observed $H_{\theta_0}(X)=H_{\theta_0}(x)$.  This conditional distribution has a density, given by 
\[ f_{h_0,\theta_0}(v_T) \propto \exp\bigl\{ -n\theta_0 v_T/2 - \cosh(v_T/2) e^{(h_0-\theta_0 v_T)/2} \bigr\}. \]
If we let $F_{h_0,\theta_0}$ denote the corresponding distribution function, then we can describe this conditional association model by
\[ T(X) = z(\theta) + F_{h_0,\theta_0}^{-1}(W), \quad W \sim {\sf Unif}(0,1). \]
If, for the P-step, we use the predictive random set $\S$ in \eqref{eq:default.prs}, then the local conditional plausibility function is
\[ \cpl_{T(x)|h_0,\theta_0}(\theta_0; \S) = 1-\bigl| 1 - 2 F_{h_0,\theta_0}(T(x)-z(\theta_0)) \bigr|. \]
A local conditional $100(1-\alpha)$\% plausibility interval for $\theta$ can be found just as before, by thresholding the plausibility function at $\alpha$.  It follows from Theorem~\ref{thm:local.valid} that these intervals will have the nominal coverage probabilities.  

For illustration, we consider a simple simulation example.  We compute the local conditional 90\% plausibility interval for $\theta$ in for 5000 Monte Carlo samples where, in each case, the true $\theta$ is sampled from $\{0.0, 0.3, 0.6, 0.9\}$.  For several values of $n$, the estimated coverage probabilities and expected lengths are compared, in Table~\ref{table:bvn}, to those of the conditional frequentist interval based on the so-called ``$r^\star$'' approximation due to \citet{bn1986} and \citet{fraser1990}, summarized nicely in \citet{reid1995, reid2003}, and a Bayesian credible interval based on Jeffreys prior.  The general message is that, compared to the other methods, the local conditional IM intervals have exact coverage for all $n$, though the intervals appear to be slightly longer on average when $n$ is small.  But when $n$ is moderate or large, there is no apparent difference in the performance.  Since one cannot hope to do much better than the Jeffreys' prior Bayes intervals for large $n$, we see that the local conditional IM results are at least asymptotically efficient, along with being valid for all $n$.  

\begin{table}
\caption{\label{table:bvn} Coverage probabilities and expected lengths of 90\% interval estimates for $\theta$ in the bivariate normal problem based on, respectively, the local conditional IM (LCIM), the $r^\star$ approach reviewed by \citet{reid2003}, and a Jeffreys prior Bayes approach.}
\centering
\fbox{
\begin{tabular}{rcccccccc}
&& \multicolumn{3}{c}{Coverage probability} && \multicolumn{3}{c}{Expected length} \\
\cline{3-5} \cline{7-9}
$n$ && LCIM & $r^\star$ & Bayes && LCIM & $r^\star$ & Bayes \\
\hline 
10 && 0.896 & 0.845 & 0.880 && 0.66 & 0.61 & 0.62 \\
25 && 0.895 & 0.867 & 0.883 && 0.42 & 0.40 & 0.41 \\
50 && 0.907 & 0.897 & 0.907 && 0.30 & 0.30 & 0.30 \\
100 && 0.903 & 0.888 & 0.896 && 0.21 & 0.21 & 0.21 \\
\end{tabular}
}
\end{table}

\subsection{A normal variance components model}
\label{SS:variance.components}

Consider the following standard two variance components model, 
\[ Y^{(g)} = (Y_{g1},\ldots,Y_{gn_g})^\top \sim \nm_{n_g}(\mu 1_{n_g}, \theta_\eps I_{n_g} + \theta_\alpha J_{n_g}), \]
independent across $g=1,\ldots,G$.  Here $G$ is the number of treatments, and $n_g$ is the number of replications under treatment $g$.  Not all $n_g$ are equal, so this is an unbalanced design.  The parameter of interest is $\theta=(\theta_\alpha,\theta_\eps)$, the variance components.  This model corresponds to the marginal distribution of the response in a simple one-way random effects model, 
\[ Y = \mu 1_n + Z\alpha + \eps, \]
where $n=\sum_{g=1}^G n_g$ is the total sample size, $Y$ is a $n$-vector obtained by stacking the $Y^{(g)}$s, $Z$ is a $n \times G$ binary matrix such that $\E(Y^{(g)} \mid \alpha_g) = (\mu + \alpha_g)1_{n_g}$, the random effects $\alpha_1,\ldots,\alpha_G$ are iid $\nm(0,\theta_\alpha)$, and $\eps$ is a $n$-vector of independent $\nm(0,\theta_\eps)$ noise.  These models are very useful in problems where variability comes from two sources.  The goal is to make inference on these two sources of variation.  

The common mean $\mu$ is a nuisance parameter, which we will eliminate with a transformation; the more general mixed-effects model case where $\mu$ is a linear function of some fixed covariates can be handled similarly.  This marginalization can be justified within the IM framework; see \citet{immarg}.  Our setup here is like that in \citet{e.hannig.iyer.2008}; the more general case, with more than two variance components, as in \citet{cisewski.hannig.2012}, shall be considered elsewhere.  

Following \citet{olsen.seely.birkes.1976}, let $K$ be a $n \times (n-1)$ matrix such that $K^\top K = I_{n-1}$ and $KK^\top = I_n - n^{-1} 1_n1_n^\top$.  Find the matrix $M = K^\top Z Z^\top K$ and let $\lambda_1 > \cdots > \lambda_L \geq 0$ be the distinct eigenvalues of $M$; let $r_\ell$ be the multiplicity of $\lambda_\ell$, $\ell=1,\ldots,L$.  Take $P=[P_1,\ldots,P_L]$ a $(n-1) \times (n-1)$ orthogonal matrix, such that $P^\top M P$ is a diagonal matrix with the eigenvalues, in their multiplicities, fall on the diagonal.  Here $P_\ell$, which corresponds to $\lambda_\ell$, is a $(n-1) \times r_\ell$ matrix.  Define
\[ X_\ell = Y^\top K P_\ell P_\ell^\top K^\top Y, \quad \ell=1,\ldots,L. \]
Then $(X_1,\ldots,X_L)$ are minimal sufficient for $\theta=(\theta_\alpha,\theta_\eps)$, and they satisfy the distributional equations
\[ X_\ell = (\lambda_\ell \theta_\alpha + \theta_\eps) U_\ell, \quad \ell=1,\ldots,L, \]
where $U_1,\ldots,U_L$ are independent, with $U_\ell \sim \chisq(r_\ell)$.  In our case of an unbalanced one-way random effects model, we know that $L$ is 1 plus the number of distinct group sample sizes $n_g$.  Thus, $L > 2$, and since the parameter of interest $\theta$ is two-dimensional, there is room to reduce the auxiliary variable down further from $L$ to 2.  To accomplish this, we shall employ the differential equation-based technique proposed above.  To make some connection to the original $Y$ sample, note that $\lambda_L=0$ and $X_L = \sum_{g=1}^G \sum_{j=1}^{n_g} (Y_{gj}-\bar Y_{g\cdot})^2$ is the usual error sum of squares.  The other $X_\ell$'s, for $\ell=1,\ldots,L-1$, are also sums of squares, but these are less familiar than $X_L$.   

To start, for a given $X=x$ and $\theta$, we can solve for $u$ in the above association:
\[ u_{x,\theta,\ell} = \frac{x_\ell}{\lambda_\ell \theta_\alpha + \theta_\eps}, \quad \ell=1,\ldots,L-1, \quad u_{x,\theta,L} = \frac{x_L}{\theta_\eps}. \]
Differentiating this expression with respect to both components of $\theta$ gives an $L \times 2$ matrix $\del u_{x,\theta} / \del\theta = \diag\{u_{x,\theta}\}W(\theta)$, where the rows of $W(\theta)$ are given by
\[ w_\ell(\theta) =  \Bigl( -\frac{\lambda_\ell}{\lambda_\ell \theta_\alpha + \theta_\eps}, \, -\frac{1}{\lambda_\ell \theta_\alpha + \theta_\eps} \Bigr), \quad \ell=1,\ldots,L-1, \quad w_L(\theta) = (0, -\theta_\eps^{-1}). \]
Choose an arbitrary localization point $\theta_0 = (\theta_{0\alpha}, \theta_{0\eps})$.  The goal is to find a function $\eta_{\theta_0}(u)$ that satisfies 
\begin{equation}
\label{eq:vc.diffeq}
\underset{(L-2) \times L}{\frac{\del \eta_{\theta_0}(u)}{\del u} \Bigr|_{u=u_{x,\theta}}} \cdot \underset{L \times L}{\diag\{u_{x,\theta}\}} \cdot \underset{L \times 2}{W(\theta)} = \underset{(L-2) \times 2}{0} \qquad \text{at $\theta=\theta_0$}. 
\end{equation}
The method of characteristics \citep{polyanin2002} suggests the logarithmic function 
\[ \eta_{\theta_0}(u)^\top = \bigl( \log u_1 , \cdots , \log u_L \bigr) \Pi(\theta_0)^\top, \]
where $\Pi(\theta_0)$ is a $(L-2) \times L$ matrix with rows orthogonal to the columns of $W(\theta_0)$.  Since $\Pi(\theta_0)W(\theta_0)$ vanishes, it is easy to check that \eqref{eq:vc.diffeq} holds for this $\eta_{\theta_0}$.  Then the corresponding $H_{\theta_0}(x)$ satisfies 
\[ H_{\theta_0}(x)^\top =  \Bigl( \log\frac{x_1}{\lambda_1 \theta_{0\alpha} + \theta_{0\eps}} , \cdots, \log\frac{x_{L-1}}{\lambda_{L-1} \theta_{0\alpha} + \theta_{0\eps}}, \log\frac{x_L}{\theta_{0\eps}} \Bigr) \Pi(\theta_0)^\top. \]
Take two orthogonal $L$-vectors which are not orthogonal to the columns of $W(\theta_0)$.  One of these vectors should be $(0,\ldots,0,1)$, so that one component of the conditional association will be free of $\theta_\alpha$.  The other vector can be, say, $(1,\ldots,1,0)$.  Then define the mapping $\tau(u)$, taking values in $\RR^2$, via the equation
\[ \begin{pmatrix} \eta_{\theta_0}(u) \\ \tau(u) \end{pmatrix} = \begin{pmatrix} \Pi(\theta_0) \\ 1 \; \cdots \; 1 \; 0 \\ 0 \; \cdots \; 0 \; 1 \end{pmatrix} \begin{pmatrix} \log u_1 \\ \vdots \\ \log u_L \end{pmatrix}. \]
This, in turn, defines the two-dimensional $(T_1,T_2)(x)$ to be used in the conditional association.  In particular, the conditional association is given by
\[ \sum_{\ell=1}^{L-1} \log X_\ell = \sum_{\ell=1}^{L-1} \log(\lambda_\ell \theta_\alpha + \theta_\eps) + \sum_{\ell=1}^{L-1} \log U_\ell, \qquad \log X_L = \log\theta_\eps + \log U_L, \]
and we set $T_1=\sum_{\ell=1}^{L-1} \log X_\ell$, $T_2=\log X_L$, $\tau(U)_1 = \sum_{\ell=1}^{L-1} \log U_\ell$, and $\tau(U)_2=\log U_L$.  Furthermore, since this representation is linear on the log scale, and $U_1,\ldots,U_L$ are independent chi-square, the conditional distribution of $\tau(U)$, given $\eta_{\theta_0}(U) = H_{\theta_0}(x)$, can be readily found numerically.  Samples from this conditional distribution, and of the corresponding predictive random set, can be obtained via Markov chain Monte Carlo.  

Details on efficient implementation of the IM-based solution to this important problem, along with comparisons with other methods, will be presented elsewhere.  Here we only give a brief illustration.  We simulate data from the one-way random effects model above, with $\theta=(1,1)$, $\mu=0$, $G=5$, and group sample sizes$(4,4,4,8,48)$; this configuration is one considered in \citet[][Section~4]{e.hannig.iyer.2008}, having moderate degree of unbalance.  A box plot of the data, in Figure~\ref{fig:vc}(a), shows evidence that suggests $\theta_\alpha > 0$.  Figure~\ref{fig:vc}(b) shows the 90\% local conditional IM plausibility region for $(\log\theta_\alpha,\log\theta_\eps)$ which, in this case, contains the true parameter value $(0,0)$.  This region is computed by simulating from the conditional distribution of $\tau(U)$, given $\eta_{\theta_0}(U)=H_{\theta_0}(x)$, for each $\theta_0$, with a random walk Metropolis--Hastings procedure.  The predictive random set used here is an ellipse in the $\tau(U)$-space, a $L_2$ generalization of the default predictive random set \eqref{eq:default.prs}, with an elasticity feature to avoid conflict \citep{leafliu2012}.  For comparison, we also display the contours of the fiducial density for $(\log\theta_\alpha,\log\theta_\eps)$ as given in \citet{e.hannig.iyer.2008}.  This indicates that the plausibility region shape is roughly consistent with the fiducial distribution, though efficiency comparisons remain to be worked out.

\begin{figure}
\centering
\fbox{
\subfigure[Box plot of the simulated data.]{\scalebox{0.5}{\includegraphics{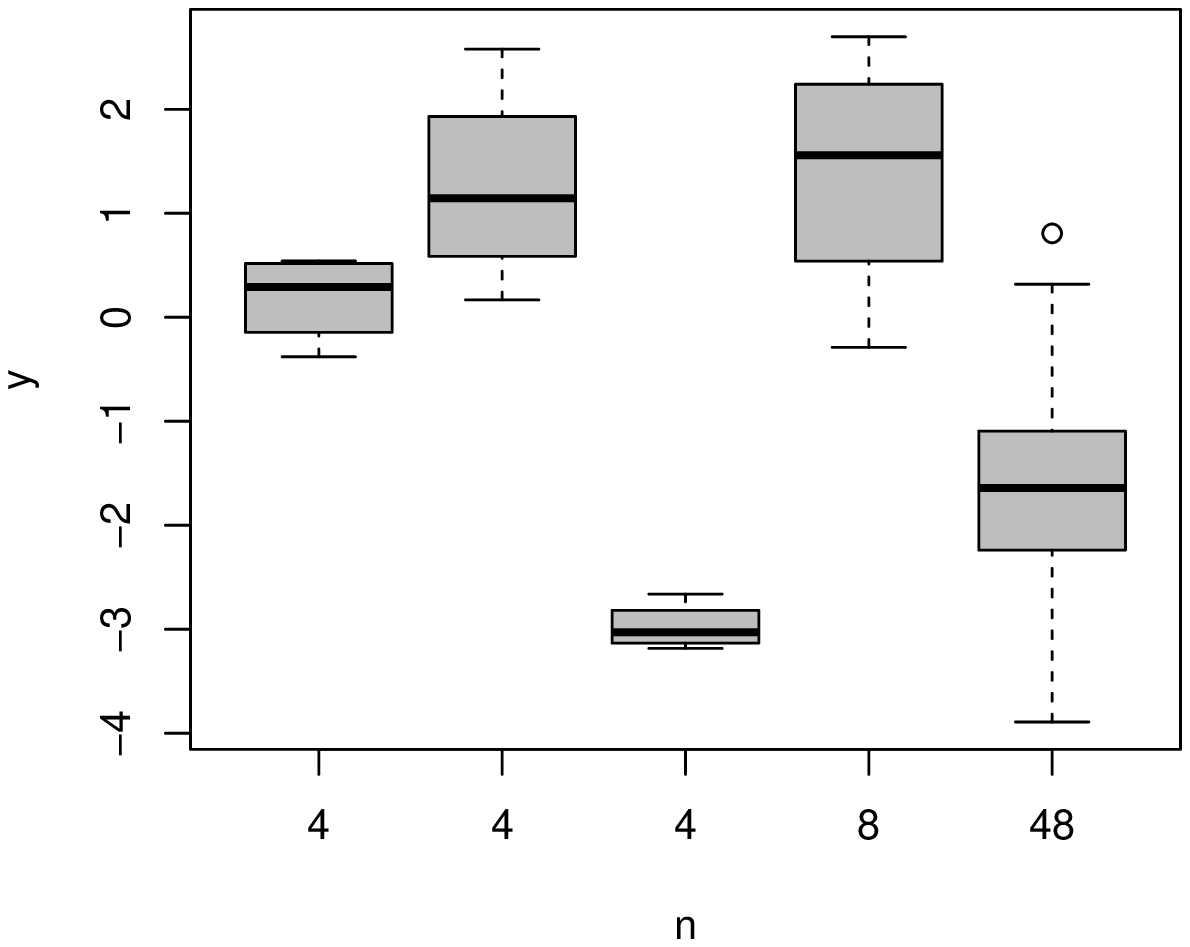}}}
\subfigure[90\% plausibility region]{\scalebox{0.5}{\includegraphics{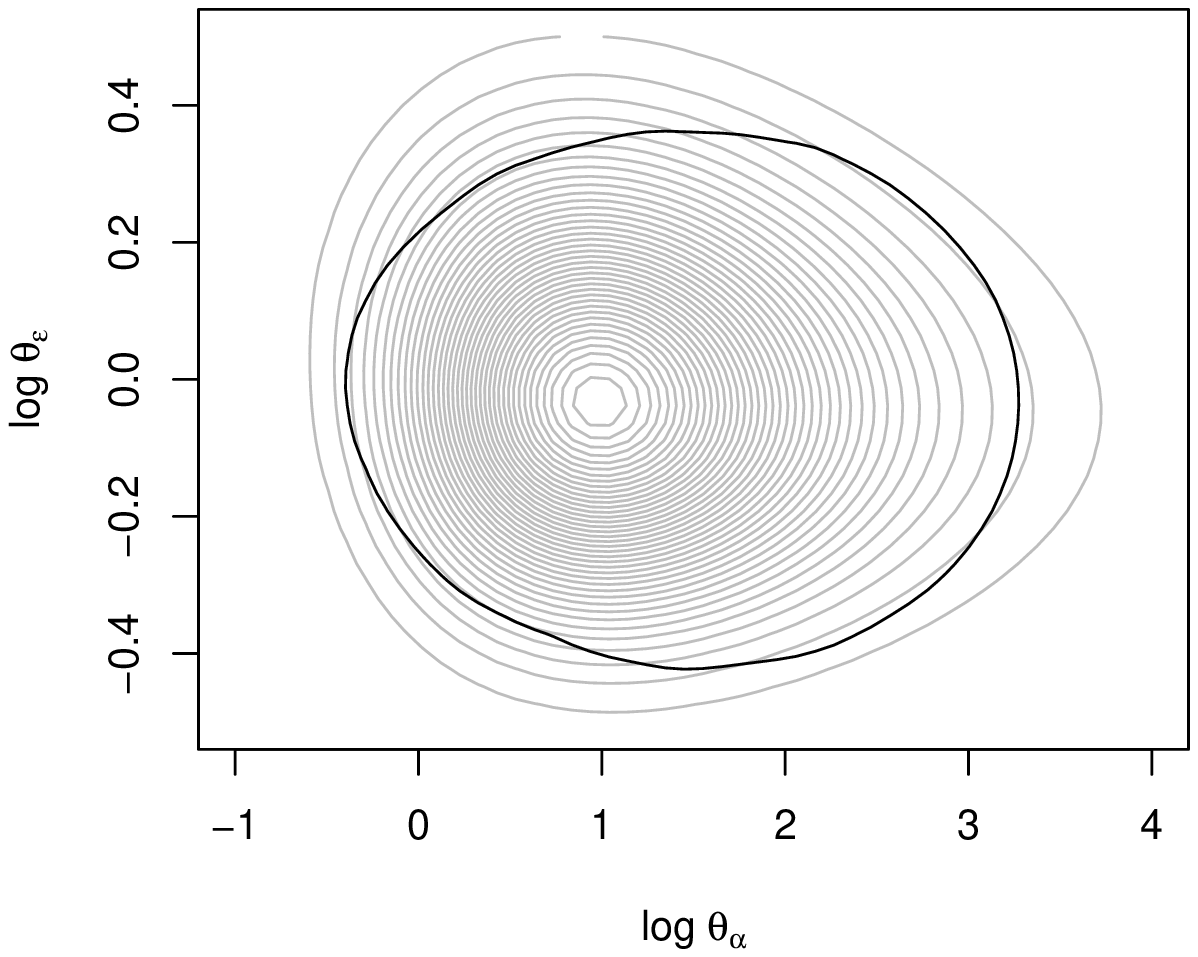}}}
}
\caption{\label{fig:vc} Variance components model results.  Panel~(b) shows the 90\% plausibility region for $(\log\theta_\alpha, \log\theta_\eps)$, along with contours of the corresponding fiducial density.  }
\end{figure}

\ifthenelse{1=1}{}{
> o$y
[[1]]
[1]  0.54141616  0.09013967  0.49293802 -0.37993473

[[2]]
[1] 0.1672459 1.0059731 1.2840267 2.5790926

[[3]]
[1] -3.184224 -3.081970 -2.974039 -2.661569

[[4]]
[1]  1.6144432  2.6974706  0.7447821  1.5032177  0.3359234  2.1095797  2.3714177
[8] -0.2886965

[[5]]
 [1] -1.39330678 -2.23811804 -2.93879057  0.30171837 -0.23902420 -3.89165357
 [7] -2.08019299 -1.83388115 -2.27917947 -1.38549644 -2.07864959 -1.36929425
[13] -1.61689209 -1.44808891 -0.63009903 -1.52609886 -2.06044181 -2.23916853
[19] -1.82055014 -0.64454870 -1.68354872 -2.48352910 -0.86196571 -1.14072965
[25] -0.98005622 -2.83419944 -0.64228204  0.80674963 -1.02605427 -3.79719287
[31] -3.56076212 -1.66729200 -2.48173891  0.04844312 -2.77707819 -1.13054611
[37] -1.50251792 -2.53070621 -1.52032184 -1.69329656 -1.10663243 -1.95309921
[43] -1.99874567 -1.35706648  0.31776222 -3.16746843 -1.86167492 -1.08127791
}

\section{Discussion}
\label{S:discuss}

This paper extends the basic IM framework laid out in \citet{imbasics} by developing an auxiliary variable dimension reduction strategy.  This reduction simultaneously accomplishes two goals.  First, it provides a suitable combination of information across samples, and we argue in Remarks~\ref{re:sufficiency} and \ref{re:bayes} in Section~\ref{SS:remarks} that Fisher's concept of sufficiency and Bayes' theorem can both be viewed as special cases of this combination of information via conditioning.  Second, this reduction makes construction of efficient predictive random sets considerably simpler.  A new differential equation technique is proposed by which an auxiliary variable dimension reduction can be found even in cases where sufficiency  fails to give a satisfactory reduction.  In addition, as our simulation results in Sections~\ref{SS:student} and \ref{SS:bvn2} demonstrate, even with a default choice of predictive random set, the conditional IMs are as good or better than those standard likelihood and Bayes methods.  This suggests that our proposed method of combining information is efficient.  We expect that the conditional IM approach, paired with the optimal predictive random sets, will have even better performance.  However, more work is needed on efficient computation of these optimal predictive random sets.  

The local conditional IMs considered in Section~\ref{S:gcim} are an important contribution.  Indeed, these tools provide a means to reduce the effective dimension even in cases where the minimal sufficient statistic has dimension greater than that of the parameter.  For example, in the variance-components problem in Section~\ref{SS:variance.components}, we identified a one-dimensional auxiliary variable to predict, even though there is no dimension reduction that can be achieved via sufficiency.  The idea of focusing on validity locally at a single $\theta=\theta_0$ itself seems to provide an improvement, this is, in fact, a special case of a more general idea.  One could measure locality by a general assertion $A$, not necessarily a singleton $A=\{\theta_0\}$.  In this way, one can develop a conditional IM that focuses on validity at a particular assertion $A$, thus extending the range of application of local conditional IMs.  Even this latter extension is a special case of a more general idea, where associations are based on generic functions of $(X,\theta,U)$, not necessarily exact formulations of the sampling model.  This new idea will be explored elsewhere.   

The examples in this paper focused on continuous distributions.  Efficient inference in discrete problems is challenging in any framework, and IMs are no different.  For nice discrete problems, e.g., regular exponential families, the IM analysis described herein can be carried out without difficulty.  However, when sufficiency considerations alone provide inadequate auxiliary variable dimension reduction, new tools are needed.  

Here, the goal was to combine information about a single quantity coming from different sources, and conditioning is shown to be the right tool.  There are other cases, however, where dimension reduction is needed because the real quantity of interest is some lower-dimensional characteristic of the full unknown parameter.  For these nuisance parameter problems, a different sort of dimension reduction is needed.  The companion paper \citep{immarg} deals with marginalization from an IM point of view.

\section*{Acknowledgments}

This work is partially supported by the U.S. National Science Foundation, grants DMS-1007678, DMS-1208833, and DMS-1208841.  The authors thank Dr.~Jing-Shiang Hwang for comments on an earlier draft, as well as the helpful suggestions given by the Editor and the anonymous Associate Editor and referees.

\appendix

\section{Proofs}
\label{S:proofs}

\begin{proof}[Proof of Proposition~\ref{prop:conditioning}]
For the given predictive random set $\S$ for $U$ in the baseline association, the corresponding random set $\Theta_x(\S)$ in the C-step can be written as 
\begin{align*}
\Theta_x(\S) & = \bigcup_{u \in \S} \{\theta: T(x) = b(\tau(u), \theta), \, H(x)=\eta(u)\} \\
& = \bigcup_{u \in \S} \bigl[ \{\theta: T(x) = b(\tau(u),\theta)\} \cap \{\theta: H(x)=\eta(u)\} \bigr] \\
& = \bigcup_{u \in {\cal R}_{H(x)}} \{\theta: T(x) = b(\tau(u), \theta)\} \\
& = \Theta_{T(x)}( \tau({\cal R}_{H(x)}) ),
\end{align*}
where ${\cal R}_{H(x)} = \S \cap \{u: \eta(u)=H(x)\}$.  Next, set $\S_{H(x)}=\tau({\cal R}_{H(x)})$ and let $\S_{H(x)}$ have the distribution it inherits from $\S$ through the mapping just described.  It is important to note that the distribution of $\S_{H(x)}$ is not a conditional distribution of, e.g., $\S$ given that $\S \cap \{u: \eta(u)=H(x)\} \neq \varnothing$, etc; it is a well-defined distribution obtained via the mapping $\S \mapsto \tau(\S \cap \{u: \eta(u)=H(x)\})$.  We have shown that $\Theta_x(\S) = \Theta_{T(x)}(\S_{H(x)})$ with $\prob_\S$-probability~1 for all $x$, so 
\[ \Theta_x(\S) \neq \varnothing \iff \Theta_{T(x)}(\S_{H(x)}) \neq \varnothing \quad \text{and} \quad \Theta_x(\S) \subseteq A \iff \Theta_{T(x)}(\S_{H(x)}) \subseteq A. \]
Since $\prob_\S\{\Theta_x(\S) \neq \varnothing\} > 0$ for all $x$, the two belief functions 
\begin{align*}
\bel_x(A;\S) & = \prob_\S\{\Theta_x(\S) \subseteq A \mid \Theta_x(\S) \neq \varnothing\}, \\
\cbel_{T(x)}(A; \S_{H(x)}) & = \prob_\S\{\Theta_{T(x)}(\S_{H(x)}) \subseteq A \mid \Theta_{T(x)}(\S_{H(x)}) \neq \varnothing\}, 
\end{align*}
are well-defined conditional probabilities, i.e., no Borel paradox issues, and, moreover, must be equal for all $x$ and all assertions $A$, as the proposition claimed.  
\end{proof}

\begin{lemma}
\label{lem:prs.valid}
Fix $h \in H(\XX)$ and take conditionally admissible $\S \sim \prob_{\S|h}$ as in Section~\ref{SS:condcred}.  Write $Q_{\S|h}(v_T) = \prob_{\S|h}\{\S \not\ni v_T\}$.  Then $Q_{\S|h}(V_T)$ is stochastically no larger than $\unif(0,1)$ for $V_T \sim \prob_{V_T|h}$.  
\end{lemma}

\begin{proof}
Just like that of Theorem~1$'$ in \citet{imbasics.c}.
\end{proof}

\begin{proof}[Proof of Theorem~\ref{thm:cond.valid}]
Take any $\theta \not\in A$ as the true value of the parameter.  Next, note that $T(X) = b(V_T,\theta)$, with $V_T \sim \prob_{V_T|h}$, characterizes the conditional distribution of $X$, given $H(X)=h$.  Since $A \subset \{\theta\}^c$, monotonicity of the belief function gives 
\[ \cbel_{T(X)|h}(A;\S) \leq \cbel_{T(X)|h}(\{\theta\}^c;\S) = \prob_{\S|h}\{\Theta_{T(X)}(\S) \not\ni \theta\} = Q_{\S|h}(V_T). \]
By Lemma~\ref{lem:prs.valid}, the right-hand side is stochastically no larger than ${\sf Unif}(0,1)$.  This, in turn, implies the same of the left-hand side $\cbel_{T(X)|h}(A;\S)$, as a function of $X \sim \prob_{X|\theta}$, given $H(X)=h$.  Therefore, 
\[ \prob_{X|\theta}\{\cbel_{T(X)|h}(A;\S) \geq 1-\alpha \mid H(X) = h\} \leq \prob\{ {\sf Unif}(0,1) \geq 1-\alpha\} = \alpha. \]
Taking supremum over $\theta \not\in A$ proves \eqref{eq:cond.valid.bel}.  
\end{proof}

\begin{proof}[Proof of Corollary~\ref{co:cond.valid}]
Since the distribution of $\S$ is free of $h$ in this case, the belief function $\cbel_{T(X)|h} \equiv \cbel_{T(X)}$ is also free of $h$.  Therefore, before taking supremum in the last line of the proof of Theorem~\ref{thm:cond.valid}, we can take expectation over $h$ to remove the conditioning, so that the validity property holds unconditionally, like in \eqref{eq:bel.valid}.  
\end{proof}

\bibliographystyle{apalike}
\bibliography{/Users/rgmartin/Dropbox/Research/mybib}

\end{document}